\documentclass[11pt]{article}
\usepackage{amssymb, amsfonts, amsthm,amsmath}
\usepackage{setspace}
\usepackage{a4wide}

\newcommand{\bw}{\mathbf{w}}

\newcommand{\eps}{\varepsilon}
\newcommand{\Pro}[1]{\mathbb{P} \left[\,#1\,\right]}
\newcommand{\Ex}[1]{\mathbb{E} \left[\, #1\,\right]}

\newcommand{\Po}{\mathsf{Po}}

\newcommand{\remove}[1]{}

\newcommand{\A}{{\ensuremath{\mathcal{A}}}}

\newcommand{\E}{\mathcal{E}}
\newcommand{\PRO}[1]{\mathbb{P}\left[\,#1\,\right]}

\newcommand{\ex}[1]{\mathbb{E} \left[\,#1\, \right]}

\renewcommand{\bw}{\mathbf{w}}
\newcommand{\bW}{\mathbf{W}}
\newcommand{\V}{\mathbf{c}}
\newcommand{\cC}{\mathsf{C}}
\newcommand{\cU}{\mathsf{U}}
\newcommand{\cK}{\mathcal{K}}
\newcommand{\cP}{\mathcal{P}}
\newcommand{\Ao}{\mathcal{A}_0}
\newcommand{\Af}{\mathcal{A}_f}

\def\RR{\mathbb{R}}

\def\PP{\mathbb{P}}
\def\EE{\mathbb{E}}

\newtheorem{theorem}{Theorem}[section]

\newtheorem{lemma}[theorem]{Lemma}

\newtheorem{claim}[theorem]{Claim}
\newtheorem{corollary}[theorem]{Corollary}

\newtheorem{proposition}[theorem]{Proposition}

\newtheorem{definition}[theorem]{Definition}

\title{Bootstrap Percolation in Inhomogeneous Random Graphs}

\author{H.~Amini\thanks{Swiss Finance Institute, EPFL, Lausanne, Switzerland.}
\and N.~Fountoulakis\thanks{School of Mathematics, University of Birmingham, United Kingdom. 
Research supported by the EPSRC Grant No. EP/K019749/1.}
\and K.~Panagiotou\thanks{Institute of Mathematics, Ludwig-Maximilians-Universit\"at, Munich, Germany.}}

\begin{document}
\maketitle

\begin{abstract}
A bootstrap percolation process on a graph with $n$ vertices is an ``infection" process evolving in rounds. Let $r \ge 2$ be fixed. Initially, there is a subset of infected vertices. In each subsequent round every uninfected vertex that has at least $r$ infected neighbours becomes infected as well and remains so forever.

We consider this process in the case where the underlying graph is an inhomogeneous random graph whose kernel is of rank one. Assuming that initially every vertex is infected independently with probability $p \in (0,1]$, we provide a law of large numbers for the size of the set of vertices that are infected by the end of the process. Moreover, we investigate the case $p = p(n) = o(1)$ and we focus on the important case of inhomogeneous random graphs exhibiting a power-law degree distribution with exponent $\beta \in (2,3)$. The first two authors have shown in this setting the existence of a critical $p_c =o(1)$ such that with high probability if $p =o(p_c)$, then the process does not evolve at all, whereas if $p = \omega(p_c)$, then the final set of infected vertices has size $\Omega(n)$. In this work we determine the asymptotic fraction of vertices that will be eventually infected and show that it also satisfies a law of large numbers.
\end{abstract}
\section{Introduction}
A \emph{bootstrap percolation process} with \emph{activation threshold} an integer $r\geq 2$
on a graph $G=G(V,E)$ is a deterministic process evolving in rounds.
Every vertex has two states: it is either \emph{infected} or \emph{uninfected} (sometimes also referred to as \emph{active} or
\emph{inactive}, respectively).
Initially, there is a subset $\Ao \subseteq V$ that consists of infected vertices, whereas every other vertex is uninfected.
Subsequently, in each round, if an uninfected vertex has at least $r$ of its neighbours infected, then it also becomes infected
and remains so forever. The process stops when no more vertices become infected, and we denote the final infected set by $\Af$.

The bootstrap percolation process was introduced by Chalupa, Leath and Reich~\cite{ChLeRe:79} in 1979 in the context of magnetic disordered systems. This process (as well as numerous variations of it) has been used as a model
to describe several complex phenomena in diverse areas, from jamming transitions~\cite{tobifi06} and
magnetic systems~\cite{sadhsh02} to neuronal activity~\cite{Am-nn, ET09} and spread of defaults in banking systems~\cite{ACM11}. Bootstrap percolation has also connections to the dynamics of the Ising model at zero temperature~\cite{Fontes02,GlauberMorris2009}. A short survey regarding applications can be found in~\cite{AdL03}.

Several qualitative characteristics of bootstrap percoloation, and in particular the dependence of the initial set $\Ao$ on the final infected set $\Af$, have been studied on a variety of graphs, such as trees~\cite{BPP06, FS08}, grids~\cite{CM02, holroyd03, BBDM2010}, lattices on the hyperbolic plane~\cite{BootHyper2013}, hypercubes~\cite{BB06}, as well as on many models of random graphs~\cite{Am-bp, balpit07, ar:JLTV10}. In particular, consider the case $r = 2$ and $G$ is the two-dimensional grid with $V = [n]^2=\{1, \dots, n\}^2$ (i.e., a vertex becomes infected if at least two of its neighbours are already infected). Then, for $\Ao \subseteq V$ whose elements are chosen independently at random, each with probability $p=p(n)$, the following sharp threshold was determined by Holroyd \cite{holroyd03}. The probability $I(n, p)$ that the entire square is eventually infected satisfies $I(n,p) \rightarrow 1$ if $\liminf_{n \rightarrow \infty} p(n) \log n > \pi^2/18$,
and  $I(n,p) \rightarrow 0$ if $\limsup_{n \rightarrow \infty} p(n) \log n < \pi^2/18$.
A generalization of this result to the higher dimensional case was proved by Balogh, Bollob\`as and Morris \cite{BBM09} (when $G$ is the 3-dimensional grid on $[n]^3$ and $r=3$) and Balogh, Bollob\`as, Duminil-Copin and Morris \cite{BBDM2010} (in general).



In this paper we study the bootstrap percolation process on inhomogeneous random graphs.
Informally, these random graphs are defined through a sequence of weights that are assigned to the vertices which, in turn, determine the probability that two vertices are adjacent. More specifically, we are interested in the case where this probability is proportional to the product of the weights of these vertices. In particular, pairs of vertices where are at least one of them has a high weight are more likely to appear as edges.

A special case of our setting is the $G(n,p)$ model of random graphs, where every edge on a set of $n$ vertices is
present independently with probability $p$. Here every vertex has the same weight. Recently, Janson, {\L}uczak, Turova and Vallier~\cite{ar:JLTV10} presented a complete analysis of the bootstrap percolation process for
various ranges of $p$. We focus on their findings regarding the range where $p = d/n$ and $d>0$ is fixed, as they are most relevant for the setting studied in this paper. In~\cite{ar:JLTV10} a law of large numbers for $|\Af|$ was shown when the density of $\Ao$ is positive, that is, when $|\Ao| = \theta n$, where $\theta \in (0,1)$.
It was further shown that when $|\Ao| = o(n)$, then typically no evolution occurs.
In other words, the density of the initially infected vertices must be positive in order for the density of the finally infected vertices to increase. This fact had been pointed out earlier by Balogh and Bollob\'as, cf.~\cite{balpit07}.
A similar behavior was observed in the case of random regular graphs \cite{balpit07}, as well as in
random graphs with given vertex degrees. These were studied by the first author in~\cite{Am-bp}, when the sum of the square of degrees scales linearly with $n$. As we shall see shortly, the random graph model we consider here is essentially a random graph with given \emph{expected} degrees.

The main result of this paper provides a law of large numbers for $|\Af|$ given $|\Ao|$ for weight sequences that satisfy fairly general and
 natural regularity conditions. We then consider weight sequences that follow a power law distribution, i.e., the proportion of vertices with
weight~$w$ scales like $w^{-\beta}$ for some~$\beta > 2$, with a particular focus on the case where $\beta \in (2,3)$. 
The parameter $\beta$ is called the \emph{exponent} of the power law. 
Note that although in this case the weight sequence has a bounded average weight, its second moment is growing with the number of vertices. Power-laws emerge in several contexts such as ranging from ecology and economics to social networks (see e.g.\ the survey of
Mitzenmacher~\cite{Mitz2004}). Already during the late 19th century Pareto observed a power law in the disrtibution of the wealth within
populations~\cite{b:p96}. In a completely different context, Lotka~\cite{ar:l26} in 1926 observed a power law distribution on the 
frequencies of scientists that are cited a certain number of times in Chemical Abstracts during the period 1910-1916.  
The article of Albert and Barab\'asi~\cite{AlbBarStatMechs} provides several examples of networks that exhibit
power law degree distributions. In fact, most of these examples exhibit power laws that have exponents between 2 and 3. 
This range of exponents is also associated with \emph{ultra-small} worlds. Chung and Lu~\cite{CL03} showed that for the model which 
we will consider in this paper, the average distance between two vertices in the largest (giant) component scales like $\log \log n$. 

In this work we extend a theorem proved by the first two authors in~\cite{ar:AmFount2012} giving a threshold function $a_c (n)= o(n)$
such that when $a(n)$ grows slower than $a_c (n)$, then with high probability no evolution occurs, but if $a(n)$ grows faster than 
$a_c (n)$, then even if $a(n)=o(n)$, the final set contains a positive fraction of the vertices. Here we determine exactly this fraction and
we show that as long as $a(n) = o(n)$, then it does not depend on $a(n)$ itself. In the rest of this section we proceed with the definition
of the random graph model that we consider and the statement of our theorems.

\paragraph{Notation}
For non-negative sequences $x_n$ and $y_n$ we
write $x_n = O(y_n)$ if there exist $N \in \mathbb{N}$ and $C > 0$ such that $x_n \leq C y_n$ for all $n \geq N$, and
$x_n = o(y_n)$, if $x_n/ y_n \rightarrow 0$, as $n \rightarrow \infty$.
We sometimes also write $x_n \ll y_n$ for $x_n = o(y_n)$.

Let $\{ X_n \}_{n \in \mathbb{N}}$ be a sequence of real-valued random variables on a sequence of probability spaces
$\{ (\Omega_n, \mathbb{P}_n)\}_{n \in \mathbb{N}, \mathcal{F}_n}$. 
If $c \in \mathbb{R}$ is a constant, we write $X_n \stackrel{p}{\rightarrow} c$ to denote that $X_n$ \emph{converges in probability to $c$}, that is, for any $\eps >0$ we have $\mathbb{P}_n (|X_n - c|>\eps) \rightarrow 0$ as $n \rightarrow \infty$. Moreover, let $\{ a_n \}_{n \in \mathbb{N}}$ be a sequence of real numbers that tends to infinity as $n \rightarrow \infty$. We write $X_n = o_p (a_n)$, if $|X_n|/a_n$ \emph{converges to 0 in probability}. If $\mathcal{E}_n$ is a measurable subset of $\Omega_n$, for any $n \in \mathbb{N}$, we say that the sequence $\{ \mathcal{E}_n \}_{n \in \mathbb{N}}$ occurs \emph{asymptotically almost surely (a.a.s.)} or \emph{with high probability (w.h.p.)} if $\mathbb{P}_n (\mathcal{E}_n) = 1-o(1)$ as
$n\rightarrow \infty$.


\section{Models and Results}\label{sec:models}


The random graph model that we consider is asymptotically equivalent to a model
considered by Chung and Lu~\cite{CL03}, and is a special case of the so-called
\emph{inhomogeneous random graph}, which was introduced by S\"oderberg~\cite{ar:s02} and
defined in its full generality by Bollob\'as, Janson and Riordan in~\cite{ar:bjr07}.

\subsection{Inhomogeneous Random Graphs with Rank-1 Kernel}

Let $n\in\mathbb{N}$ and consider the vertex set~$[n] := \{1, \dots, n\}$. Each vertex~$i$ is assigned a positive
weight~$w_i(n)$, and we will write~$\bw=\mathbf{w}(n) = (w_1(n), \dots, w_n(n))$. We will often suppress the
dependence on $n$, whenever it is obvious from the context. For convenience, we will assume that $w_1 \leq w_2 \leq \cdots \leq w_n$.
For any~$S \subseteq [n]$, set
\[
	W_S(\mathbf{w}) := \sum_{i\in S} w_i.
\]
In our random graph model the event of including the edge~$\{i,j\}$ in the
resulting graph is independent of the inclusion of any other edge, and
its probability equals
\begin{equation}
\label{eq:pijCL}
	p_{ij}(\mathbf{w}) = \min\left\{\frac{w_iw_j}{W_{[n]}(\mathbf{w})},
1\right\}.
\end{equation}
This model was studied by Chung et al.\ for fairly general choices of~$\mathbf{w}$,
who studied in a series of papers~\cite{CL02,CL03} several typical properties of the resulting graphs,
such as the average distance between two randomly chosen vertices that belong to the same component or the component size distribution. We will refer to this model as the \emph{Chung-Lu} model, and we shall
write~$CL(\mathbf{w})$ for a random graph in which each possible edge~$\{i,j\}$ is included
independently with probability as in~\eqref{eq:pijCL}. Moreover, we will suppress the dependence on
$\mathbf{w}$, if it is clear from the context which sequence of weights we refer to.

Note that in a Chung-Lu random graph the weights (essentially) control the
\emph{expected} degrees of the vertices. Indeed, if we ignore the minimization
in~\eqref{eq:pijCL}, and also allow a loop at vertex~$i$, then the expected
degree of that vertex is~$\sum_{j=1}^n w_iw_j/W_{[n]} = w_i$.

\subsection{Regular Weight Sequences}

Following van der Hofstad~\cite{book:vdH}, for any $n\in \mathbb{N}$ and
any sequence of weights $\mathbf{w}(n)$ let
\[
	 F_n(x) = n^{-1} \sum_{i=1}^n \mathbf{1}[w_i(n) \leq x], \ \ \forall x\in [0,\infty)
\]
be the empirical distribution function of the weight of a vertex chosen uniformly at random.
We will assume that $F_n$ has a certain structure.
\begin{definition}
\label{def:regcond}
We say that $(\mathbf{w}(n))_{n \ge 1}$ is \emph{regular}, if it has the following 
properties.
\begin{itemize}
  \item {\bf[Weak convergence of weight]} There is a distribution function
$F:[0,\infty)\to [0,1]$ such that for all $x$ at which $F$ is continuous
$\lim_{n\to\infty}F_n(x) = F(x)$.
\item {\bf[Convergence of average weight]} Let $W_n$ be a random variable with
distribution function $F_n$, and let $W_F$ be a random variable with
distribution function $F$. Then $\lim_{n\to\infty}\Ex{W_n} = \Ex{W_F} < \infty$.
\item {\bf[Non-degeneracy]} There is a $x_0 \in \mathbb{R}^+$ such that $F_n(x) = 0$ for all $x \in [0,x_0)$ and  $n \in \mathbb{N}$.
\end{itemize}
\end{definition}
The regularity of~$(\mathbf{w}(n))_{n\ge 1}$ guarantees two important properties. Firstly,
the weight of a random vertex is approximately distributed as a random variable that follows a certain distribution.
Secondly, this variable has finite mean and it is easy to see that the associated Chung-Lu random graph has bounded average degree with high probability. The third property in Definition~\ref{def:regcond} is a minor restriction guaranteeing that no vertex has a vanishing expected degree and is added for convinience in order to simplify several of our technical considerations.

At many places in our arguments it will be important to select vertices randomly according to their weight, i.e. the probability to choose $i \in [n]$ equals $w_i / W_{[n]}(\mathbf{w})$. This is the so-called \emph{size-biased} distribution and we denote by $W_F^*$ a random variable with this distribution. A straightforward calculation shows that for every bounded continuous function~$f$ 
\begin{equation} \label{eq:Size_biased}
\Ex {f(W_F^*)} = {\Ex {W_F f(W_F)} \over \Ex{W_F}}.
\end{equation}
\subsection{Results}

The main theorem of this paper gives a law of large numbers for the size of $\Af$ when $\Ao$ has positive density in the case where the underlying random graph is a Chung-Lu random graph with a regular weight sequence.
Let $\psi_r (x)$ for $x \geq 0$ be equal to the probability that a Poisson-distributed random variable with
parameter $x$ is at least $r$, i.e.,
$$\psi_r (x):= \PP(\Po(x)\geq r) = \sum_{j \ge r} e^{-j}\,  {x^j}/{j!}.$$
Let $X$ be a non-negative random variable and $p \in [0,1]$. For any $r\geq 1$ and $y \in \mathbb{R}^+$ set
$$ f_r (y; X,p) = (1-p) \ex {\psi_r (Xy)} + p - y.$$
\begin{theorem} \label{thm:FinalInfection}
Let $(\mathbf{w}(n))_{n \geq 1}$ be regular with limiting distribution function $F$. Consider
the bootstrap percolation process on $CL(\mathbf{w})$ with activation threshold $r\geq 2$, where $\Ao \subseteq [n]$ includes any
vertex independently with fixed probability $p \in (0,1)$. Let $\hat{y}$ be the smallest positive solution of
\begin{equation}
\label{eq:defhaty}
f_r(y;W_F^*,p) = 0.
\end{equation}
Assume also that $f_r'(\hat{y};W_F^*,p) < 0$.
Then
\begin{equation}
\label{eq:convergenceGeneral}
n^{-1}{|\Af|} \stackrel{p}{\rightarrow} (1-p)\ex {\psi_r (W_F\hat{y})}+p, \ \mbox{as $n \rightarrow \infty $}.
\end{equation}
\end{theorem}
We remark that a solution $\hat{y}$ to~\eqref{eq:defhaty} always exists because $f_r(y;W_F^*,p)$ is continuous, $f_r(0;W_F^*,p) > 0$ and $f_r(1;W_F^*,p) \leq 0$. Note that the conclusion of our results is valid only if
$f_r'(\hat{y};W_F^*,p) < 0$. This does not happen only if
$$\EE\left[\frac{e^{-\hat{y}W_F^*}(W_F^*\hat{y})^r}{r!}\right] = \frac{\hat{y}}{(1-p)r},$$
and for such (rather exceptional) weight sequences we expect a different behavior.
Moreover, we show that (c.f.\ Lemma~\ref{lem:power_law_derivative}) if the weight sequence has power law distribution with exponent between 2 and 3, this case will not happen (i.e., $f_r'(\hat{y};W_F^*,p) < 0$ always).

Intuitively, the quantity $\hat{y}$ represents the limit of the probability that a random vertex is becomes infected. 
The fixed-point equation $f_r(y;W_F^*,p) = 0$, whose solution $\hat{y}$ is, effectively says that a vertex is infected if 
either it is externally infected (which occurs with probability $p$) or (if not, which occurs with probability $1-p$) it has at least 
$r$ infected neighbours. The latter is a Poisson-distributed random variable with parameter equal to $W_F^* \hat{y}$. 
The first factor essentially states the fact that a vertex becomes some other vertex's neighbour with probability proportional to the 
latter's weight, whereas it is infected with probability approximately $\hat{y}$.   

We will now see an extension of the above theorem to the case where $p$ is not anymore bounded away from 0. Under certain conditions
the above theorem can be transferred to this case simply setting $p=0$.  These conditions ensure that a positive but rather
small fraction of the vertices become infected and this effectively corresponds to taking a $p$ that is in fact bounded away from 0 but small. 

\subsection{Power-law Weight Sequences}

Our second result focuses on an important special case of weight sequences, namely those following a power law distribution. This is described by the following condition.
\begin{definition}
\label{def:Distr}
We say that a regular sequence~$( \mathbf{w}(n) )_{n \ge 1}$ follows \emph{a power law with exponent~$\beta$},
if there are~$0< c_1< c_2$, $x_0 >0$ and $0 < \zeta \leq {1/(\beta -1)} $ such that for all~$x_0 \le x \le n^{\zeta}$
\[
	c_1 x^{-\beta + 1}\le 1 - F_n(x) \le c_2 x^{-\beta + 1},
\]
and $F_n(x) = 0$ for $x < x_0$ and $F_n(x) = 1$ for $x > n^{\zeta}$.
Moreover, for any $x > x_0$ we have for some $c>0$
$$ \lim_{n \to \infty} F_n(x) = F(x) = 1 - c x^{-\beta + 1}. $$
\end{definition}
Note that the above definition implies that for $\zeta > {1/(\beta -1)}$, we have $n(1-F_n(n^\zeta)) = 0$, since
$1-F_n(n^\zeta) \leq c_2 n^{-\zeta(\beta-1)}= o(n^{-1})$. So it is natural to assume that $\zeta  \leq {1/(\beta -1)}$.

A particular example of a power-law weight sequence is given in~\cite{CL03}, where the authors choose~$w_i = d ({n}/{(i + i_0)})^{1/(\beta - 1)}$ for some $d>0$. This results typically in a graph with a power law degree sequence with exponent~$\beta$, average degree~$O(d)$, and maximum degree proportional to~$({n}/{i_0})^{1/(\beta - 1)}$, see also~\cite{book:vdH} for a detailed discussion. When $\beta \in (2,3)$, these random graphs are also characterized as \emph{ultra-small worlds}, due to the fact that the typical distance of two vertices that belong to the same component is $O(\log \log n)$, see~\cite{CL03,book:vdH}.

Theorem~\ref{thm:FinalInfection} addresses the case where the initial set $\Ao$ has positive density. Our second result is complementary
and considers the setting where $p = p(n) = o(1)$, with a particular focus on the case where the exponent of the power law is in $(2,3)$
Assume that $\Ao$ has density $a(n)/n$. In~\cite{ar:AmFount2012} the first two authors determined a function $a_c(n)$ (see the statement of the next theorem)
such that when $a(n) = o (a_c (n))$, then a.a.s.\ $|\Ao| = |\Af|$, whereas if $a(n) = \omega(a_c(n))$ but $a(n) = o(n)$, then a.a.s.\ $|\Af| > \eps n$, for some $\eps >0$. We refine this result using the proof of Theorem~\ref{thm:FinalInfection} and determine the fraction of vertices that belong to $\Af$.
\begin{theorem} \label{thm:power_law}
Let $(\mathbf{w}(n))_{n \geq 1}$ be regular following a power law with exponent $\beta\in(2,3)$ and with 
$\zeta$ that satisfies ${r-1 \over 2r - \beta + 1} < \zeta \leq {1\over \beta -1}$.
Let $$a_c (n) = n^{(r(1-\zeta) + \zeta (\beta -1)-1)/ r}.$$
Consider the bootstrap percolation process on $CL (\mathbf{w})$ with activation threshold $r\geq 2$.
Assume that $\Ao$ is a random subset of $[n]$ where each vertex is included independently with probability $a(n)/n$. Then, if $a(n) = \omega(a_c (n))$ and $a(n) = o(n)$
$$ n^{-1}{|\Af|} \stackrel{p}{\rightarrow} \ex {\psi_r (W_F \hat{y})} \ \mbox{as $n \rightarrow \infty $}, $$
where $\hat{y}$ is the smallest positive solution of
$$ y = \ex {\psi_r (W_F^* y)}.$$
\end{theorem}
Let us remark here that the (rescaled) size of the final set does not depend on $|\Ao|$. We should also point out that the lower bound 
on $\beta$ has its origins at the proof of the main theorem in~\cite{ar:AmFount2012}. For $\zeta \leq {r-1 \over 2r - \beta + 1}$, the 
authors of~\cite{ar:AmFount2012} identified two functions $a_c^{-}(n) \ll a_c^{+}(n) = o(n)$ such that if $a(n) \gg a_c^+(n)$, then 
$|\Af| > \eps n$, for some $\eps >0$, but if $a(n) \ll a_c^- (n)$, then a.a.s. $|\Ao|= |\Af|$. In fact, the proof of the above theorem is 
such that it also holds for $a(n) \gg a_c^+ (n)$. More generally, the above theorem holds as long as the initial density is
such that a.a.s most vertices of weight that is larger some big constant become infected.

\subsection{Outline}

The proofs of Theorems~\ref{thm:FinalInfection} and~\ref{thm:power_law} are based on a \emph{finitary} approximation of the weight sequence $\bw (n)$. In the following section we construct a sequence of weight sequences having only a finite number of weights and that ``approximate'' the initial sequence in a certain well-defined sense. Thereafter, we show the analogue of Theorem~\ref{thm:FinalInfection} for finitary sequences; this is Theorem~\ref{thm:FinalInfectionI} stated below. The proof of Theorem~\ref{thm:FinalInfectionI} is based on the so-called \emph{differential equation} method, which was developed by
Wormald~\cite{Worm95,Worm99}, and is used to keep track of the evolution of the bootstrap percolation process through the exposure of the neighbours of each
infected vertex. Such an exposure algorithm was also applied in the homogeneous setting~\cite{ar:JLTV10}. Of course,
the inhomogeneous setting imposes significant obstacles. We close the paper with the proof of some  rather technical results, which transfer the condition on the derivative that appears in the
statement of Theorem~\ref{thm:FinalInfection} in the finitary setting.

\section{Finitary Weight Sequences}
In this section we will consider, what we call, \emph{finitary} weight sequences on $[n]$ that are suitable approximations of an arbitrary weight sequence $\bw(n)$. As a first step we are going to ``remove'' all weights from $\mathbf{w}$ that are too large in the following sense. Suppose that  $\bw (n)$ is regular and that the corresponding sequence of empirical distributions
 converges to $F$.
For $\gamma > 0$ let
$$C_{\gamma} = C_\gamma(F) = \inf \{x  \mid F(x) \geq 1-\gamma \}.$$
Then, as $n \rightarrow \infty$, the following facts are immediate consequences. Let $\cC_{\gamma} = \cC_{\gamma}(n,F)$ be the set of vertices in $[n]$ with weight at least $C_\gamma(F)$. 
\begin{enumerate}
\item[1.] If the infimum in the definition of $C_\gamma(F)$ is attained then $|\cC_{\gamma}(n,F)|/n \rightarrow \gamma$.
\item[2.] We have that
$$ n^{-1}{W_{\cC_{\gamma}(n,F)} (\bw(n)) } \rightarrow \int_{C_{\gamma}}^{\infty} x dF(x) =:
W_{\gamma}(F), $$
where the latter is the Lebesque-Stieltjes integral with respect to $F$.
\item[3.] The assumption $\ex {W_F} = d < \infty$ implies that $\PRO {W_F > x } = o(1/x)$ as $x \rightarrow \infty$. Thus
\begin{equation}
\label{eq:convpr}
C_\gamma(F) \PRO {W_F > C_\gamma(F)} \rightarrow 0, \ \mbox{as $\gamma \downarrow 0$}.
\end{equation}
We will be using this observation at several places in our proofs.
\end{enumerate}
We will approximate a regular $(\bw(n))_{n \ge 1}$ by a sequence where most vertices have their weights within a finite set of values and moreover the weights are bounded by $2C_\gamma(F)$ (cf.~\cite{book:vdH} where a similar approach is followed in a different context).

\begin{definition} \label{def:FinitarySequence}
Let $\ell \in \mathbb{N}$ and $\gamma \in (0,1)$. Let $n'=n'(n) \in \mathbb{N}$ be an increasing function of $n$.
We say that a regular weight sequence
\[
	(\bW^{(\ell, \gamma)}(n))_{n \ge 1 }
	= \left(W_1^{(\ell, \gamma)}(n),\dots, W_{n'}^{(\ell, \gamma)}(n)\right)_{n \ge 1}
\]
is a $(\ell, \gamma)$-\emph{discretisation} of a regular $(\bw(n))_{n \ge 1}$ if the following conditions are satisfied. Let $x_0 > 0$ be such that $F_n(x) = 0$ for all $x < x_0$ and $n \in \mathbb{N}$ and let $F$ be the limiting distribution of $(\bw(n))_{n \ge 1}$. Then there is a $p_\ell \in \mathbb{N}$, and real numbers $\gamma_1, \dots, \gamma_{p_\ell} \in (0,1)$ such that $\sum_{i=1}^{p_\ell} \gamma_i = 1 - \gamma$ and real weights $x_0 \le W_1 \le \dots \le W_{p_\ell} \le C_\gamma(F)$ such the following hold.
\begin{enumerate}
\item[1.] There is a partition of $[n]\setminus \cC_{\gamma}(F)$ with $p_\ell$ parts, denoted by $\cC_1(n),\ldots, \cC_{p_\ell}(n)$ such that $|\cC_i(n)| = (1+o(1)) \gamma_i n $ for all $1\le i \le p_\ell$.
\item[2.] For all $1 \le i \le p_\ell$ and for all $j \in \cC_i(n)$  we have $W_j^{(\ell, \gamma)}(n') = W_i$.
\item[3.] Let $\cC_\gamma'(n) :=  [n']\setminus \cup_{i=1}^{p_\ell} \cC_i(n)$. Then
$C_{\gamma}(F) \leq W_j^{(\ell, \gamma )}(n') \leq 2 C_{\gamma}(F)$ for all $j \in \cC_\gamma'(n)$ .
\end{enumerate}
Moreover, as $n \to\infty$
\begin{enumerate}
\item[4.] There is a $0 \le \gamma' < \gamma + 2W_\gamma(F)/C_\gamma(F)$ such that
$ n^{-1}{|\cC_\gamma'(n)|} \rightarrow \gamma' . $
\item[5.] There is a $0 \le W_\gamma' \le 5 W_\gamma(F)$ such that
$n^{-1}{W_{\cC_\gamma'(n)}(\bW^{(\ell,\gamma)}(n))} \rightarrow  W_\gamma'.$
\end{enumerate}
\end{definition}
The upper bounds in \emph{4.}\ and \emph{5.}\ are tailored to the proof of
Theorem~\ref{thm:FinalInfection}. Note that in the previous definition no requirement is made on the magnitude of the $W_i$s, and thus $\bW^{(\ell,\gamma)}$ might look very different from $\bw$. The next definition quantifies when a $(\ell,\gamma)$-discretisation is ``close'' to a given regular $(\bw(n))_{n \ge 1}$ with limiting distribution function $F$.
\begin{definition} \label{def:F-conv}
Let $(\bw(n))_{n \ge 1}$ be regular and let $F$ be its limiting distribution function. Let $\rho:[0,1] \to \mathbb{R}^+$ be such that $\lim_{\gamma\to0}\rho (\gamma ) = 0$.
A sequence $( (\bW^{(\ell, \gamma)}(n))_{n \ge 1} )_{\ell \in \mathbb{N}}$ of discretisations of $(\bw(n))_{n \ge 1}$ is called $F$-\emph{convergent with error $\rho$} if there exists $\gamma_0>0$ and $L_1(\gamma) \in \mathbb{N}$ such that for any $\gamma < \gamma_0$ and  $\ell > L_1 (\gamma)$
\begin{enumerate}
\item
$\sup_{x\in[x_0,C_\gamma(F)]} |F^{(\ell, \gamma)}(x)- F(x)| < 2(\gamma + W_\gamma(F) /C_\gamma(F))$ and \item $\left|\int_{0}^{\infty}x dF^{(\ell,\gamma)}(x) - d \right| < \rho(\gamma)$,  where $d=\int_{0}^{\infty}x dF(x)=\mathbb{E}[W_F]$.
\end{enumerate}
\end{definition}
Let us write~$F^{(\ell, \gamma)}_{n}$ for the empirical distribution function of the weight of a random vertex in $CL(\bW^{(\ell,\gamma)}(n))$. By assumption, $F^{(\ell, \gamma)}_{n}$ converges to a function $F^{(\ell, \gamma)}$. It follows that the corresponding random variable $W_{F^{(\ell,\gamma)}}$, which we denote for brevity by $U^{(\ell, \gamma)}$, is  such that for $1 \le i \le p_\ell$
\begin{equation}
\label{eq:probs}
	\PRO{W_{F^{(\ell,\gamma)}} = W_i}
	= \PRO{U^{(\ell,\gamma)} = W_i} =  {\gamma_i \over 1- \gamma + \gamma'}=:\gamma_i'.
\end{equation}

Now let $F^{*(\ell,\gamma)}$ denote the cumulative distribution function of
the $U^{(\ell,\gamma)}$ size-biased distribution, cf.~\eqref{eq:Size_biased}. Let also $F^*$ denote that of $W_F^*$.
The conditions of Definition~\ref{def:F-conv} imply the following technical statement that we will use later in our proof, which states that $F^{*(\ell,\gamma)}$ is close to $F^*$ almost everywhere in the interval $[x_0, C_\gamma(F)]$. In particular, let $\mathcal{D}_{\gamma, \ell}$ denote the set of discontinuities of $F^{*(\ell,\gamma)}$ in the closed interval
$[x_0,C_\gamma(F)]$ and let $\mathcal{D}$ denote the set of discontinuities of $F^*$. We set
$$\mathcal{D}_\gamma := \mathcal{D} ~\cup~ \bigcup_{\ell \in \mathbb{N}} \mathcal{D}_{\gamma,\ell}.$$
This is a countable set and therefore it is of measure zero.
We will show that the $L_\infty$-norm of the difference between $F^*$ and $F^{*(\ell,\gamma)}$ on $[x_0,C_\gamma(F)] \setminus \mathcal{D}_\gamma$ can be bounded by a vanishing (as $\gamma \rightarrow 0$) function of $\gamma$.
\begin{lemma}\label{lem:size_biased_approx}
Let $(\bw(n))_{n \ge 1}$ be regular and let $F$ be its limiting distribution function. Let $( (\bW^{(\ell, \gamma)}(n))_{n\ge 1})_{\ell \in \mathbb{N}}$ be $F$-{convergent with error $\rho$}. Then there exists a $\rho_1:\mathbb{R}^+\to\mathbb{R}^+$ such that $\lim_{\gamma\to0}\rho_1 (\gamma ) = 0$ for which the following holds. There is $\gamma_1 >0$ such that for any $0 < \gamma < \gamma_1$
and any $\ell$ sufficiently large (depending on $\gamma$ only)
$$ | F^{*(\ell,\gamma)}(x) - F^*(x)| < \rho_1 (\gamma)
\qquad
\text{for all }  x \in [x_0,C_\gamma(F)] \setminus \mathcal{D}_\gamma.$$
\end{lemma}
\begin{proof}
Let $x \in [x_0,C_\gamma(F)] \setminus \mathcal{D}_\gamma$.
For $t > 0$ let
\begin{equation*}
	h^{(t)}(y) := h_x^{(t)} (y) :=
	\begin{cases}
		1- \exp((y-x)t), & y < x,\\
		0, & y \geq x
		\end{cases}.
\end{equation*}
We will use this function as a continuous approximation of $\mathbf{1}_{\{y < x \}}$.
Indeed, for any $y \in \mathbb{R}$ we have $h^{(t)}(y) \to \mathbf{1}_{\{ y < x \}}$ as $t \rightarrow \infty$. The dominated
convergence theorem yields for any $\gamma, \ell$
\begin{equation} \label{eq:t-convergence}
\begin{split}
\int_{0}^{\infty} h^{(t)}(y)  dG(y)  & \stackrel{t\rightarrow \infty}\rightarrow
\int_{0}^{\infty} \mathbf{1}_{\{ y < x \}}  dG (y) = G (x-),
\quad \text{where } G \in \{F^{*(\ell,\gamma)}, F^*\}.
\end{split}
\end{equation}
The definition of the size-biased random variable (cf.\ \eqref{eq:Size_biased}) implies that
\begin{equation*}
\begin{split}
\int_{0}^{\infty} h^{(t)}(y)  dF^{*(\ell,\gamma)} (y)  &= {\ex {U^{(\ell,\gamma)} h^{(t)} (U^{(\ell,\gamma)})} \over \ex {U^{(\ell,\gamma)}}}
\mbox{ and }
\int_{0}^{\infty} h^{(t)}(x)  dF^{*} (y) = {\ex {W_F  h^{(t)} (W_F )} \over d}.
\end{split}
\end{equation*}
We are going to show that there is $\gamma_1 >0$ such that for any $0 < \gamma < \gamma_1$,  any $\ell$ large enough depending only on $\gamma$  and any $t>0$ we have: for $x \in [x_0, C_\gamma] \setminus \mathcal{D}_\gamma$
\begin{equation}
\label{eq:h_approx}
\left| {\ex {W_F  h^{(t)} (W_F )}  \over d} - {\ex {U^{(\ell,\gamma)}  h^{(t)} (U^{(\ell,\gamma)} )}
\over \ex {U^{(\ell,\gamma)}} } \right| < {2 \over d} \big( \rho (\gamma ) +  4(\gamma C_\gamma(F)  + W_\gamma(F) ) \big).
\end{equation}
With this fact at hand the proof of Lemma~\ref{lem:size_biased_approx} can be completed as follows. By \eqref{eq:t-convergence} we can choose $t$ large enough so that, say,
\begin{equation*}
\begin{split}
\left|\int_{0}^{\infty} h^{(t)}(y)  dG(y) -
\int_{0}^{\infty} \mathbf{1}_{\{ y < x \}}  dG(y)\right| < \rho (\gamma),
\quad \text{where } G \in \{F^{*(\ell,\gamma)}, F^*\}.
\end{split}
\end{equation*}
Thus, by~\eqref{eq:h_approx} and the triangle inequality we obtain
$$ | F^{*(\ell,\gamma)} (x) - F^* (x) | < 2 \rho (\gamma ) + {2 d^{-1}} \left( \rho (\gamma ) +  4(\gamma C_\gamma  +
W_\gamma ) \right) =: \rho_1 (\gamma).$$
Note that $\rho_1 \to 0$ when $\gamma\to0$, due to~\eqref{eq:convpr}. This completes the proof of Lemma~\ref{lem:size_biased_approx}. We finish with the proof of~\eqref{eq:h_approx}.
We write
\begin{equation} \label{eq:expectations_together}
\begin{split}
\ex {W_F  h^{(t)} (W_F )} = \int_{0}^{x } y h^{(t)} (y) dF(y)
\mbox{ and }
 \ex {U^{(\ell,\gamma)}  h^{(t)} (U^{(\ell,\gamma)} )} =
\int_{0}^{x} y h^{(t)} (y) dF^{(\ell,\gamma)}(y).
\end{split}
\end{equation}
Note that $y h_x^{(t)} (y)$ is differentiable (and therefore continuous) for any $x_0 \leq y < x$ and the modulus of
its derivative is bounded in this interval. Hence, it has bounded total variation. These facts allow us to use the integration-by-parts
formula for the Lebesque-Stieltjes integral. Abbreviating $g(y) = y h^{(t)} (y)$, since $F$ is right-continuous, we obtain
\begin{equation} \label{eq:1st_Int}
\begin{split}
\int_{0}^{x} g(y) dF(y)
& =  F(x) x  h^{(t)} (x) - 0 \cdot h^{(t)} (0) F(0-) - \int_{0}^{x} F(y) d g (y) \\
&= F(x) x  h^{(t)} (x) - \int_{0}^{x} F(y) d g (y).
\end{split}
\end{equation}
Similarly, we obtain
\begin{equation} \label{eq:2nd_Int}
\begin{split}
\int_{0}^{x}  g (y) dF^{(\ell,\gamma)}(y) =
 F^{(\ell,\gamma)} (x) x  h^{(t)} (x ) -
\int_{0}^{x } F^{(\ell,\gamma)}(y) d g (y).
\end{split}
\end{equation}
The first part of Definition~\ref{def:F-conv} implies that if $0 < \gamma < \gamma_0$ and $\ell > L_1 (\gamma)$, then 
$$
|F^{(\ell, \gamma)} (x) - F(x)| < 2(\gamma + W_\gamma(F) / C_\gamma(F)). $$
Let us abbreviate $y(\gamma)= 2(\gamma + W_\gamma / C_\gamma)$. With this notation,~\eqref{eq:expectations_together}, \eqref{eq:1st_Int} and \eqref{eq:2nd_Int} together yield
\begin{equation*}
\label{eq:diffnumer}
\begin{split}
	\left| \ex {W_F  h^{(t)} (W_F )} - \ex {U^{(\ell,\gamma)}  h^{(t)} (U^{(\ell,\gamma)} )} \right|
	& \le
	y(\gamma) xh^{(t)}(x) + \int_{0}^{x} |F(y) - F^{(\ell,\gamma)} (y) |  d g (y) \\
	& \le
	y(\gamma) C_\gamma(F) + y(\gamma) (g(x) - g(0)) \\
	& \le
	2y(\gamma)C_\gamma(F).
\end{split}
\end{equation*}
The second part of Definition~\ref{def:F-conv} implies that for any $\ell$ large enough (depending only on $\gamma$)
\begin{equation*} \label{eq:den}
\left| \ex {U^{(\ell,\gamma)}} - d \right| < \rho (\gamma).
\end{equation*}
Using that for $a,b,c,d > 0$ we have that $|\frac{a}b - \frac{c}d| = |\frac{ad-cb}{bd}|\le \frac{|a-c|}{\min\{b,d\}} + \frac{a|b-d|}{bd}$ we get the estimate
\begin{equation*}
\begin{split}
& \left| {\ex {W_F  h^{(t)} (W_F )}  \over d} - {\ex {U^{(\ell,\gamma)}  h^{(t)} (U^{(\ell,\gamma)} )}
\over \ex {U^{(\ell,\gamma)}} } \right| \\
 \le & \frac{\left| \ex {W_F  h^{(t)} (W_F )} - \ex {U^{(\ell,\gamma)}  h^{(t)} (U^{(\ell,\gamma)} )} \right|}{ \min\{d,\ex {U^{(\ell,\gamma)}}\}} + \frac{\ex {W_F  h^{(t)} (W_F )}\left |d- \ex {U^{(\ell,\gamma)}}\right|}{d \ex {U^{(\ell,\gamma)}}}.
\end{split}
\end{equation*}
Let $0 < \gamma_1 \le \gamma_0$ be such that for any $0 < \gamma  < \gamma_1$ we have $\rho(\gamma) < d/2$. Then for all such $\gamma$ the first term in the previous expression is bounded by $4y(\gamma)C_\gamma(F)/d$ and since $\ex {W_F  h^{(t)} (W_F )} \le d$ the second term is bounded by $2\rho(\gamma)/d$; \eqref{eq:h_approx} follows.
\end{proof}
For technical reasons we consider a slightly different definition of the random graph model that we denote by $CL' ( \bW^{(\ell, \gamma)} )$. In this modified model the edge probabilities are proportional
to the product of the weights of the vertices, except that the normalizing factor is not equal to the sum of the weights in
$\bW^{(\ell, \gamma)}$, but it is equal to $W_{[n]}(\bw(n))$, that is, the edge $\{i,j\}$ is contained in $CL' ( \bW^{(\ell, \gamma)} )$ with probability
\[
	p_{ij}(\bW^{(\ell, \gamma)}(n), \bw(n)) = \min\left\{ \frac{w_i^{(\ell,\gamma)} w_i^{(\ell,\gamma)}}{W_{[n]}(\bw)} ,1\right\}.
\]
The next theorem quantifies the number of the finally infected vertices when the weight sequence is a discretisation of a given regular $(\bw(n))_{n \ge 1}$. It is general enough so that it can be used in the proof of Theorem~\ref{thm:power_law} as well.
\begin{theorem} \label{thm:FinalInfectionI}
Let $r \geq 2$, $\gamma \in (0,1)$. Let $(\bw(n))_{n \ge 1}$ be regular and let $F$ be its limiting distribution function. 
There exists $\gamma_2>0$ such that for $\gamma\in(0,\gamma_2)$  and for any $\delta  \in (0,1)$ 
there is a subsequence $\mathcal{S} :=\{\ell_k\}_{k\in \mathbb{N}}$ with the property that for any $\ell \in  \mathcal{S}$ 
the following holds.
Let $( (\bW^{(\ell, \gamma)}(n))_{n\ge 1})_{\ell \in \mathbb{N}}$ be $F$-{convergent with error $\rho$}. Moreover, assume that $f_r'(\hat{y};W_F^*,p) <0$ (cf. Theorem~\ref{thm:FinalInfection}).

Assume that initially all vertices of $CL' (\bW^{(\ell, \gamma)})$ that belong to
$\cC_{\gamma}'(n)$ are infected, whereas each vertex in $\cC_i(n)$ is infected independently with probability $p\in [0,1)$, for each $i=1,\ldots, p_\ell$. Let $\Af^{(\ell,\gamma)}$ denote the set of vetices in $[n'] \setminus \cC_{\gamma}'(n)$ that become eventually infected during a bootstrap percolation process with activation threshold $r$. 
Then with probability $1-o(1)$
$$ n^{-1}{|\Af^{(\ell,\gamma)}| } = (1\pm \delta) \left( (1-p) \ex{W_F \hat{y}} + p \right). $$
\end{theorem}

\subsection{Proof of Theorem~\ref{thm:FinalInfection}}

Given a regular $(\bw(n))_{n \ge 1}$, 
Theorem~\ref{thm:FinalInfection} follows from Theorem~\ref{thm:FinalInfectionI} by constructing an $F$-convergent sequence $( (\bW^{(\ell, \gamma)}(n))_{n\ge 1})_{\ell \in \mathbb{N}}$. We first describe our construction and prove some properties of it, and then proceed with the proofs of our main results.

\subsubsection{The construction of approximating weight sequences}

Let $(\mathbf{w}(n))_{n \ge 1}$ be regular and consider the limiting distribution function $F$. For $\gamma \in (0,1)$, recall that if
$\gamma \in F([0,\infty))$, then $F(C_{\gamma}) = 1 - \gamma$, where $C_\gamma = C_\gamma (F)$. 
We assume that $\gamma$ has this property.
Recall also that from Definition~\ref{def:regcond} there is a positive real number $x_0$ such that $F(x) = 0$ for  $x < x_0$.
For any $x>0$ the symbol $F(x+)$ ($F(x-)$, respectively) will denote the right (left, resp.) limit of $F$ at $x$. Of course, $F$ is 
right-continuous and, therefore, $F(x+)=F(x)$. 

Let $d_1 < d_2 < \cdots$ be the set of discontinuities of $F$ in 
$[x_0, C_\gamma)$ -- this is a countable set (possibly finite).
These $d_i$s incur a natural partition of $[x_0 , C_\gamma )$ into half-open intervals
$D_i:= [d_i,d_{i+1})$ for $i\ge0$, where $d_0 :=x_0$. Let $c_i := d_{i+1} - d_i + F (d_{i+1}) - F(d_{i+1}-) $. In other words,
$c_i$ is the length of the interval $D_i$ together with the magnitude of the $(i+1)$st discontinuity.
Let $c_{i_1} \geq c_{i_2} \geq \cdots$ be the ordering of the $c_i$s  according to their
size and let $k_j$ be the $j$th largest size that appears in this ordering.
We set $K_j := \{ i  \ : \ c_i = k_j \}$ -- this is the set of indices $i$ such that $c_i$ has the $j$th largest size in the above ordering.
Note that $\sum_j k_j ~|K_j| < \infty$.

For any $\ell \in \mathbb{N}$, consider the set of indices $\cup_{j=1}^\ell K_j$; we assume that these are
$j_1 < \cdots < j_{s_\ell}$.
Consider now the partition $$\mathcal{P}_\ell' := \{ [x_0,d_{j_1}), [d_{j_1},d_{j_1+1}), [d_{j_1+1}, d_{j_2}), \ldots,
[d_{j_{s_\ell+1}},C_\gamma) \}.$$
Let $L=L(\ell)$ be the minimum natural number such that if we further partition each $[d_{j_i},d_{j_i+1})$ in $L$ equal half-open intervals
$[d_{j_i}=y_0, y_1),[y_1,y_2),\ldots, [y_{L-1},y_{L}=d_{j_i+1})$, then
$F(y_{j+1}) - F(y_i) < 1/ \ell$.
Also, note that for the remaining parts the quantities $F(d_{j_1}-) - F(x_0), F(d_{j_2}-) - F(d_{j_1+1}), \ldots, F(C_\gamma-) -
F(d_{j_{s_\ell} + 1})$ are bounded by $\sum_{j > \ell} k_j~|K_j|$.
We let $\mathcal{P}_\ell$ be the refinement of $\mathcal{P}_\ell'$, where we include the above parts for each interval
$[d_{j_i},d_{j_i+1})$. Let $p_\ell$ be the total number of parts in $\mathcal{P}_\ell$ and let
$I_i := [W_i^-, W_i^+)$, for $i=1,\ldots, p_\ell$, denote the $i$th part (note that each part is an interval).
We let
\begin{equation*}
	\eps_\ell : = \max \left\{1/\ell, \sum_{j > \ell} k_j~|K_j|\right\}.
\end{equation*}
This quantity bounds $F(W_i^+) - F(W_i^-)$, that is, for all $i=1,\ldots, p_\ell$
\begin{equation}
\label{eq:eps_def}
F(W_i^+) - F(W_i^-) < \eps_\ell.
\end{equation}

Given this partition and the weight sequence $\bw (n)$, for each $n \geq 1$ we define two (discretised) weight sequences
$\bW^{(\ell, \gamma)+} (n')$ and $\bW^{(\ell,\gamma)-} (n'')$ on the sets $[n']$ and $[n'']$, respectively, as follows. The partition
$\mathcal{P}_\ell$ gives rise to a partition of $[n] \setminus \cC_{\gamma}$, where for each $i =1,\ldots, p_\ell$ we have
$\cC_i = \{j \ : \ w_j(n) \in I_i \}$.
We denote this partition by $\cP_{n,\ell,\gamma}$ and we let this be the associated partition of
$\bW^{(\ell,\gamma)+} (n')$  and $\bW^{(\ell, \gamma)-} (n'')$. In particular,

\medskip

\noindent
-- for each $i = 1,\ldots, p_\ell$ and for each $j \in \cC_i$, we set
$$W_j^{(\ell, \gamma)-} (n) := W_i^-, \ \mbox{and} \ W_j^{(\ell, \gamma)+} (n') := W_i^+.$$

\medskip

\noindent
-- consider the random subset of $\cC_\gamma$, in which every element of $\cC_\gamma$ is included independently with probability $p$.
An application of the Chernoff bounds implies that a.a.s. this has size at least $\lfloor p|\cC_\gamma| - n^{2/3} \rfloor=:k_-$.
Consider a set of vertices $\cC_\gamma^-=\{ v_1,\ldots, v_{k_-} \}$ which is disjoint from $[n]$.
We identify with $[n'']$ the set $\left( \cup_{i=1}^{p_\ell} \cC_i \right)  \bigcup \cC_\gamma^-$, with the assumption that those vertices
which belong to $\cup_{i=1}^{p_\ell} \cC_i$ retain their labels.  It follows that $n'' = (1-\gamma + p\gamma)n (1+o(1))$.

For any $j \in \cC_\gamma^- = [n''] \setminus \cup_{i=1}^{p_\ell} \cC_i$ we set  $W_j^{(\ell, \gamma)-} (n) := C_\gamma$.
Note that
$$ \lim_{n \rightarrow \infty} {|\cC_\gamma^-|\over n} = p \gamma,$$
and if $W_{\cC_\gamma^-} (\bW^{(\ell,\gamma)-})$ denotes the total weight of these vertices, then this satisfies
$$ \lim_{n \rightarrow \infty} {W_{\cC_\gamma^-} (\bW^{(\ell,\gamma)-})  \over n} = p\gamma C_\gamma =:W_\gamma^-. $$

\medskip

\noindent
-- for any vertex $j \in \cC_\gamma$ such that $w_j(n) \geq 2 C_\gamma$ we consider $r_j:=2 \lfloor {w_j(n) \over C_\gamma}\rfloor$
copies of this vertex each having weight $2C_\gamma$, which we label as $v_{j1},\ldots, v_{jc_j}$.
For each such $j$ we let $\eps_j(n)= {w_j(n) \over C_\gamma} - \lfloor {w_j (n) \over C_\gamma} \rfloor$ and
we set $R = \lceil 2 \sum_{j \ : \ w_j(n) \geq 2 C_\gamma} \eps_j (n) \rceil$.
If $j \in \cC_\gamma$ is
such that $C_\gamma \leq  w_j(n) < 2 C_\gamma$, then we introduce a single copy $v_{j1}$ having weight equal to $w_j$ (in other
words $r_j=1$).

We let $\cC_\gamma^+$ be the set that is the union of these copies together with a set of $R$ vertices which we denote by
$\mathcal{R}$ (disjoint from the aforementioned sets) each having weight $2 C_\gamma$:
$$ \cC_\gamma^+ := \mathcal{R} \cup \bigcup_{j \in \cC_\gamma} \{v_{j1},\ldots, v_{jr_j} \}. $$
Let $n' = \left| \left( \cup_{i=1}^{p_\ell} \cC_i \right)  \bigcup \cC_\gamma^+ \right|$ and identify the set $[n']$ with the vertices
in $\left( \cup_{i=1}^{p_\ell} \cC_i \right)  \bigcup \cC_\gamma^+$, under the assumption that the vertices in
$\cup_{i=1}^{p_\ell} \cC_i$ retain their labels.
We will use the symbol $\cC_\gamma^+$ to denote the set $[n'] \setminus \left( \cup_{i=1}^{p_\ell} \cC_i \right)$. In other words, the
set  $\cC_\gamma^+$ consists of the replicas of the vertices in $\cC_\gamma$, as these were defined above, together with the set of
vertices corresponding to $\mathcal{R}$. This completes the definition of $\bW^{(\ell, \gamma)+} (n')$.

Note that
\begin{equation*}
\begin{split}
|\cC_\gamma^+| &=\sum_{j \ : \  C_\gamma \leq w_j <  2C_\gamma} 1 +  \sum_{j \ : \ w_j \geq  2C_\gamma }
2 \lfloor  {w_j \over C_\gamma} \rfloor +  R \\
&= \sum_{j \ : \  C_\gamma \leq w_j <  2C_\gamma} 1 + 2 \sum_{j \ : \ w_j \geq  2C_\gamma } {w_j \over C_\gamma}  + e(n),
\end{split}
\end{equation*}
with $0\leq e(n) < 1$,
whereby it follows that as $n \rightarrow \infty$
\begin{equation} \label{eq:n'_bound}
 {|\cC_\gamma^+| \over n} \rightarrow \PRO{C_\gamma \leq W_F < 2C_\gamma} + 2 {\ex{\mathbf{1}_{\{ W_F \geq 2C_\gamma \}}W_F}
\over C_\gamma}=:\gamma^+
< \gamma + 2 {W_\gamma \over C_\gamma},
\end{equation}
Hence, as $n \rightarrow \infty$
\begin{equation} \label{eq:weight_bound}
\begin{split}
{W_{\cC_\gamma^+} (\bW^{(\ell,\gamma)+})  \over n} \rightarrow
& \ex{\mathbf{1}_{\{ C_\gamma \leq W_F < 2C_\gamma \}}W_F} + 4 \ex{\mathbf{1}_{ \{ W_F \geq 2C_\gamma \} }W_F}
=:W_\gamma^+ \\
& \leq \ex{\mathbf{1}_{\{ C_\gamma \leq W_F \} }W_F} + 4 \ex{\mathbf{1}_{ \{ W_F \geq C_\gamma \}}W_F } =  5 W_\gamma.
\end{split}
\end{equation}

We denote by $U_{n}^{(\ell,\gamma)-}$ and $U_{n'}^{(\ell, \gamma)+}$ the weight in $\bW^{(\ell,\gamma)-} (n)$  and
$\bW^{(\ell, \gamma)+} (n')$ of a uniformly chosen vertex from $[n]$ and $[n']$, respectively.
Also, we let $F_n^{(\ell, \gamma)-}, F_{n'}^{(\ell, \gamma)+}$  denote their distribution functions.
Note that both $F_n^{(\ell, \gamma)-}, F_{n'}^{(\ell, \gamma)+}$ converge pointwise as $n\rightarrow \infty$
to the functions $F^{(\ell, \gamma)-}, F^{(\ell, \gamma)+}$, respectively, where

\medskip

\noindent
--for each $i = 1,\ldots, p_\ell$ and for each $x \in I_i$ which is a point of continuity of $F$, we set
$$F^{(\ell, \gamma)-} (x) := {F(W_i^-)\over 1-\gamma + p \gamma}, \ \mbox{and} \ F^{(\ell, \gamma)+} (x) =
 {F(W_i^+) \over 1- \gamma + \gamma^+}. $$

\medskip

\noindent
-- for any $x\geq C_\gamma$ we have $F^{(\ell,\gamma)-}(x) = 1$;

\medskip
\noindent
-- for any $C_\gamma \leq x < 2C_\gamma$ which is a point of continuity of $F$ we have
\begin{equation} \label{eq:tail_value} F^{(\ell,\gamma)+}(x) = {F( x ) \over 1-\gamma+ \gamma^+},\end{equation} whereas for
$x\geq 2 C_\gamma$ we have $F^{(\ell,\gamma)+}(x) = 1$.

We will now verify that both weight sequences are $F$-convergent with a certain error $\rho$, which we give explicitly.
For any $x \in I_i$ we have
\begin{equation} \label{eq:F_approx}
\begin{split}
& \left| F^{(\ell,\gamma)+}(x) - F(x) \right| < \left| {F(W_i^+) \over 1- \gamma + \gamma^+} - F(W_i^-) \right| \\
&= \left| {F(W_i^+) \over 1- \gamma + \gamma^+}  - {F(W_i^-) \over 1- \gamma + \gamma^+}
+ {F(W_i^-) \over 1- \gamma + \gamma^+} - F(W_i^-) \right| \\
& \leq \left|  {F(W_i^+) - F(W_i^-) \over 1- \gamma + \gamma^+}  \right|
+ \left| {F(W_i^-) \over 1- \gamma + \gamma^+} - F(W_i^-)  \right| \\
&\stackrel{(\ref{eq:eps_def})}{\leq} 
{\eps_\ell \over 1 - \gamma + \gamma^+} + F(W_i^-) {\gamma^+ - \gamma \over 1- \gamma + \gamma^+} \leq
{\eps_\ell \over 1 - \gamma + \gamma^+} +  {\gamma^+ - \gamma \over 1- \gamma + \gamma^+} <
{3\over 2} {\gamma^+ - \gamma \over 1- \gamma + \gamma^+},
\end{split}
\end{equation}
for any $\ell$ sufficiently large (depending on $\gamma$ only).
Similarly, for any $\ell$ sufficiently large (depending on $\gamma$) and $x \in I_i$ we have
\begin{equation}\label{eq:F_approx-}
\begin{split}
\left| F^{(\ell,\gamma)-}(x) - F(x) \right| < {\left| F(W_i^+)  - F(W_i^-) \right|\over 1- \gamma +p \gamma } 
\stackrel{(\ref{eq:eps_def})}{<}
{\eps_\ell \over 1 - \gamma+ p\gamma} <
{3\over 2} {\gamma^+ - \gamma \over 1- \gamma + \gamma^+}.
\end{split}
\end{equation}
Furthermore, since $F^{(\ell,\gamma)+}$ is constant (and equal to 1) for $x \geq 2 C_\gamma$ we have
\begin{equation*}
\begin{split}
 \int_{0}^{\infty}x dF^{(\ell,\gamma)+}(x) = \int_{0}^{2C_\gamma}x dF^{(\ell,\gamma)+}(x)
+ \int_{2C_\gamma}^{\infty}x dF^{(\ell,\gamma)+}(x) = \int_{0}^{2C_\gamma}x dF^{(\ell,\gamma)+}(x).
\end{split}
\end{equation*}
Using the integration-by-parts formula for the Lebesque-Stieltjes integral we can write the latter as
\begin{equation} \label{eq:expansion}
\begin{split}
\int_{0}^{2C_\gamma}x dF^{(\ell,\gamma)+}(x) &= 2C_\gamma F^{(\ell,\gamma)+}(2C_\gamma+) - 0 \cdot
F^{(\ell,\gamma)+}(0-)  -
\int_{0}^{2C_\gamma} F^{(\ell,\gamma)}(x)dx \\
& = 2C_{\gamma} - \int_{0}^{2C_\gamma} F^{(\ell,\gamma )+}(x)dx.
\end{split}
\end{equation}
We will approximate the above integral using (\ref{eq:F_approx}). For $\ell$ large enough we have
\begin{equation} \label{eq:temp_approx_I}
\begin{split}
&\left|\int_{0}^{2C_\gamma} F^{(\ell,\gamma )+}(x)dx - \int_{0}^{2C_\gamma} F (x)dx \right| \leq \\
&\left| \int_{0}^{C_\gamma} F^{(\ell,\gamma )+}(x)dx - \int_{0}^{C_\gamma} F (x)dx \right|
+ \left|\int_{C_\gamma}^{2C_\gamma} F^{(\ell,\gamma )+}(x)dx - \int_{C_\gamma}^{2C_\gamma} F (x)dx  \right| \\
&\leq \int_{0}^{C_\gamma}\left|F^{(\ell,\gamma)+}(x) - F(x) \right|dx +
\int_{C_\gamma}^{2C_\gamma}\left|F^{(\ell,\gamma)+}(x) - F(x) \right|dx \\
&\stackrel{(\ref{eq:F_approx}), (\ref{eq:tail_value})}{\leq} {3\over 2} {\gamma^+ - \gamma \over 1- \gamma + \gamma^+} C_\gamma
+ {\gamma^+ - \gamma \over 1- \gamma + \gamma^+} C_\gamma =
{5 \over 2} {\gamma^+ - \gamma \over 1- \gamma + \gamma^+} C_\gamma \\
& < 3 \gamma^+ C_\gamma \stackrel{(\ref{eq:n'_bound})}{<}  3\gamma C_\gamma + 6 W_\gamma.
\end{split}
\end{equation}
Applying again the integration-by-parts formula for the Lebesque-Stieltjes integral we have
\begin{equation} \label{eq:temp_approx_II}
\begin{split}
\int_{0}^{2C_\gamma} F (x)dx &= 2C_\gamma F(2C_\gamma+) - 0\cdot F(0-) - \int_{0}^{2C_\gamma} x dF(x) \\
&= 2C_\gamma (1- \PRO {W_F > 2C_\gamma})  - \int_{0}^{2C_\gamma} x dF(x).
\end{split}
\end{equation}
Hence (\ref{eq:temp_approx_I}) and (\ref{eq:temp_approx_II}) imply that
\begin{equation*}
\begin{split}
\left| \int_{0}^{2C_\gamma} F^{(\ell,\gamma )+}(x)dx + \int_{0}^{2C_\gamma} x dF(x) - 2 C_\gamma \right| <
2C_\gamma \PRO {W_F > 2 C_\gamma} + 3 \gamma C_\gamma + 6 W_\gamma,
\end{split}
\end{equation*}
whereby using (\ref{eq:expansion}) we have
\begin{equation*}
\begin{split}
\left| \int_{0}^{2C_\gamma}x dF^{(\ell,\gamma)+}(x) - \int_{0}^{2C_\gamma} x dF(x)  \right| <
2C_\gamma \PRO {W_F > 2 C_\gamma} + 3 \gamma C_\gamma + 6 W_\gamma.
\end{split}
\end{equation*}
But also
$$ \left| \int_{0}^{2C_\gamma} x dF(x)  - \int_{0}^{\infty} x dF(x)   \right|  \leq \int_{2C_\gamma}^{\infty} x dF(x) =
\ex {\mathbf{1}_{W_F > 2 C_\gamma} W_F } . $$
But $\int_{0}^{\infty} x dF(x) = d$ and therefore,
\begin{equation*}
\begin{split}
\left| \int_{0}^{2C_\gamma}x dF^{(\ell,\gamma)+}(x) - d \right| & <\\
 &2C_\gamma \PRO {W_F > 2 C_\gamma} + 3 \gamma C_\gamma + 6 W_\gamma +\ex {\mathbf{1}_{W_F > 2 C_\gamma} W_F}.
\end{split}
\end{equation*}
We set
$$\rho (\gamma ) : =  2C_\gamma \PRO {W_F > 2 C_\gamma} + 3 \gamma C_\gamma + 6 W_\gamma +\ex {\mathbf{1}_{W_F > 2 C_\gamma} W_F} . $$
Using similar estimates (cf.~(\ref{eq:F_approx-})), we can also show that
$$ \left| \int_{0}^{2C_\gamma}x dF^{(\ell,\gamma)-}(x) - d \right|  <\rho (\gamma).  $$
The above findings can be summarized in the following lemma.
\begin{lemma}\label{lem:size_biased}
As $n\rightarrow \infty$, we have $$U_{n}^{(\ell,\gamma)-} \stackrel{d}{\rightarrow} U^{(\ell, \gamma)-} \ \mbox{and} \
U_{n'}^{(\ell, \gamma)+} \stackrel{d}{\rightarrow} U^{(\ell, \gamma)+}$$
where $U^{(\ell, \gamma)-}$ and $U^{(\ell,\gamma)+}$ are random variables whose distribution functions are
$F^{(\ell,\gamma)-}, F^{(\ell,\gamma)+}$, respectively.
Furthermore, for any $\gamma \in (0,1)$ there exists $L_1(\gamma)$ such that for any $\ell > L_1(\gamma)$
$$ \left\| F^{(\ell,\gamma)-} - F \right\|_{\infty[x_0,C_\gamma]}, \left\| F^{(\ell,\gamma)+} - F \right\|_{\infty[x_0,C_\gamma]}
< {3\over 2} {\gamma^+ - \gamma \over 1- \gamma + \gamma^+}.$$
Also, for any such $\ell$ we have
$$\left| \int_{0}^{\infty}x dF^{(\ell,\gamma)+}(x)  - d \right|, \left| \int_{0}^{\infty}x dF^{(\ell,\gamma)-}(x)  - d
\right| < \rho ( \gamma). $$

\end{lemma}

\subsubsection{Bounds on $|\Af|$} For a subset $S \subseteq [n]$, let $\Af (S)$ denote the final set of infected vertices in $CL (\bw)$ assuming that $\Ao = S$. With this notation we have of course that $\Af = \Af (\Ao)$. We also set $\Af^-(S)$ to be the set of infected vertices in $CL' (\bW^{(\ell,\gamma)-})$, respectively, assuming that the initial set
is $S \cap [n'']$.
Finally, for a subset $S \subseteq [n']$ let $\Af^+(S)$  the final set of infected vertices on $CL' (\bW^{(\ell,\gamma)+})$.
We will show the following.
\begin{claim} \label{clm:stoch_sandwich}
Let $p\in(0,1)$. Assume that $\Ao$ is a random subset of $[n]$ where each vertex is included with probability $p$ independently of any other vertex. Then there is a coupling space on which a.a.s.
\begin{equation}
\label{eq:couplingAf}
|\Af^- (\Ao \cup \cC_\gamma^- )|\leq |\Af| \leq |\Af^+ (\Ao \cup \cC_\gamma^+)|.
\end{equation}
\end{claim}
\begin{proof}
As $\Ao$ is formed by including every vertex in $[n]$ independently with probability $p$, it follows that a.a.s.\ at least $k_-$ elements of $\cC_\gamma$ become initially infected. We identify exactly $k_-$ of them with the
set $\cC_\gamma^-$. Note that for each $i \in [n]$ we have
$W_i^{(\ell,\gamma)-}(n) \leq w_i(n)$. This implies that for each pair $i,j \in [n'']$ of distinct vertices, the probability that these are
adjacent is smaller in $CL'(\bW^{(\ell,\gamma)-})$ compared to that in $CL (\mathbf{w})$. Hence, there is coupling space on which
$$ CL'(\bW^{(\ell,\gamma)-}) \subseteq CL(\mathbf{w}), $$
and the first inequality in~\eqref{eq:couplingAf} follows. The second inequality follows from a slightly more involved argument. Let $j \in \cC_\gamma$ be such that $w_j (n) \geq 2C_\gamma$ and let $k \in \cup_{i=1}^\ell \cC_i $.
The probability that $k$ is adjacent to $j$ in $CL(\mathbf{w})$ is equal to $w_k w_j / W_{[n]}$.
Also, the probability that $k$ is adjacent to at least one of the copies of $j$ in $[n']$ in the random graph
$CL'(\bW^{(\ell,\gamma)+})$ is
$$ 1 - \left( 1- {2w_k C_\gamma \over W_{[n]}} \right)^{2 \lfloor {w_j / C_\gamma}\rfloor}. $$
Assume that we show that for $n$ sufficiently large we have that for any $k \in \cup_{i=1}^{p_\ell} \cC_i $ and any $j \in \cC_\gamma$
\begin{equation}  \label{eq:aim_proof}
{w_k w_j \over W_{[n]}} \leq 1 - \left( 1- {2w_k C_\gamma \over W_{[n]}} \right)^{2 \lfloor {w_j / C_\gamma}\rfloor}.
\end{equation}
Moreover, assume that every vertex in $\cC_\gamma'$ is among those vertices that are initially infected.
Now, observe that there is coupling space in which we have
\begin{equation} \label{eq:stoch_inclusion}
CL(\mathbf{w})[\cup_{i=1}^{p_\ell}\cC_i] \subseteq CL'(\bW^{(\ell,\gamma)+})[\cup_{i=1}^{p_\ell}\cC_i].
\end{equation}
This is the case, since for any $k \in \cup_{i=1}^{p_\ell}\cC_i$
we have $w_k(n) \leq W_k^{(\ell,\gamma)+}(n)$.
Consider a vertex $k \in \cup_{i=1}^{p_\ell}\cC_i$ and now let $j \in \cC_\gamma$.
Now, Inequality (\ref{eq:aim_proof}) implies that the
probability that $k$ is adjacent to $j$ in $CL(\mathbf{w})$ is at most the probability that $k$ is adjacent to at least one of the
copies of $j$ in $[n']$ within $CL'(\bW^{(\ell,\gamma)+})$. Thereby, it follows that the number of neighbours
of $k$ in $\cC_\gamma$ in the random graph $CL(\mathbf{w})$ is stochastically dominated by the size of the neighbourhood of $k$
in $\cC_\gamma'$ in the random graph $CL'(\bW^{(\ell,\gamma)+})$. This observation together with (\ref{eq:stoch_inclusion}) imply
that
$$ |\Af (\Ao \cup \cC_\gamma)| \leq_{st} |\Af^+(\Ao \cup \cC_\gamma^+)|.$$
But also,
\begin{equation*} |\Af| \leq_{st} |\Af (\Ao \cup \cC_\gamma)|.
\end{equation*}
The second stochastic inequality of the claim follows from the above two inequalities. It remains to show (\ref{eq:aim_proof}). Using the Bonferroni inequalities we have
\begin{equation} \label{eq:Bonferroni}
\begin{split}
 &1 - \left( 1- {2w_k C_\gamma \over W_{[n]}} \right)^{2 \lfloor {w_j / C_\gamma}\rfloor}
 	\geq   2 \lfloor {w_j \over C_\gamma}\rfloor {2w_k C_\gamma \over W_{[n]}} - 2 \left( {w_j / C_\gamma} \right)^2 ~{4 w_k^2 C_\gamma^2 \over W_{[n]}^2}.
\end{split}
\end{equation}
But
$$  2 \lfloor {w_j \over C_\gamma}\rfloor {2w_k C_\gamma \over W_{[n]}} \geq
2 \left({w_j \over C_\gamma} -1 \right) {2w_k C_\gamma \over W_{[n]}}
= 2 {w_j \over C_\gamma} \left( 1- {C_\gamma \over w_j} \right) {2w_k C_\gamma \over W_{[n]}}
\stackrel{w_j /C_\gamma \geq 2}{\geq}
{2 w_k w_j \over W_{[n]}}.$$
Substituting this lower bound into (\ref{eq:Bonferroni}) we obtain
$$1 - \left( 1- {2w_k C_\gamma \over W_{[n]}} \right)^{2 \lfloor {w_j \over C_\gamma}\rfloor} \geq
{2 w_k w_j \over W_{[n]}} - {8 w_k^2 w_j^2 \over W_{[n]}^2} =
 {2 w_k w_j \over W_{[n]}} \left(1 - {4 w_k w_j \over W_{[n]}} \right) >  {w_k w_j \over W_{[n]}},$$
 for $n$ sufficiently large, as $w_k < C_\gamma$ and $w_j = w_j(n) = o(n)$ (uniformly for all $j$) but $W_{[n]} = \Theta (n)$.
\end{proof}
We will now apply Theorem~\ref{thm:FinalInfectionI} to the random variables that bound $|\Af|$ in Claim~\ref{clm:stoch_sandwich}. Theorem~\ref{thm:FinalInfectionI} implies that there exists $\gamma_2 >0$ satisfying the following:
for any $\gamma < \gamma_2$ and any $\delta \in (0,1)$ there exists an infinite set of natural numbers
$\mathcal{S}^1$ such that for every $\ell \in \mathcal{S}^1$ with probability $1-o(1)$
\begin{equation} \label{eq:UpperBound}
\begin{split}
n^{-1}{|\Af^{+} (\Ao \cup \cC_\gamma^+) |} \leq
 (1+ \delta) ((1-p) \ex{ W_F \hat{y}} + p),
\end{split}
\end{equation}
and an infinite set of natural numbers
$\mathcal{S}^2$ such that for every $\ell \in \mathcal{S}^2$ with probability $1-o(1)$
\begin{equation} \label{eq:LowerBound}
\begin{split}
n^{-1}{|\Af^- (\Ao \cup \cC_\gamma^- ) |} \geq  (1-\delta) ((1-p) \ex{ W_F \hat{y}} + p).
\end{split}
\end{equation}
Hence, Claim~\ref{clm:stoch_sandwich} together with (\ref{eq:UpperBound}) and (\ref{eq:LowerBound}) imply the following a.a.s.
bounds on the size of $\Af$:
\begin{equation*}
\begin{split}
n^{-1} {|\Af| } = (1\pm \delta) ((1-p) \ex{ W_F \hat{y}} + p),
\end{split}
\end{equation*}
whereby Theorem~\ref{thm:FinalInfection} follows.

\subsection{Proof of Theorem~\ref{thm:power_law}}

Let us assume that $\Ao$ is randomly selected, including each vertex independently with probability $a(n)/n$, where
$a(n) \gg a_c (n)$ but $a(n) = o(n)$ (cf. Theorem~\ref{thm:power_law} for the definition of the function $a_c(n)$).
For $\eps \in (0,1)$ let $\Ao^{(\eps)}$ denote a random subset of $[n]$ where
each vertex is included independently with probability $\eps$. If $n$ is large enough, then $\Ao$ can be coupled with $\Ao^{(\eps)}$,
that is, there is a coupling space in which $\Ao \subseteq \Ao^{(\eps)}$.
The following stochastic upper bound can be deduced as in Claim~\ref{clm:stoch_sandwich}.
\begin{claim} \label{clm:UpperDomination}
For any $\eps \in (0,1)$ and any $\gamma >0$, if $n$ is large enough, then
\begin{equation*}
|\Af| \leq_{st} |\Af (\Ao^{(\eps)} \cup \cC_\gamma )| \leq_{st} |\Af^+ (\Ao^{(\eps)} \cup \cC_\gamma^+ )|.
\end{equation*}
\end{claim}

We will now deduce a stochastic lower bound on $|\Af|$.
For $C>0$, let $\cK_C$ denote the set of vertices having weight at least $C$ in ${\mathbf w}$.
In~\cite{ar:AmFount2012} the first two authors prove that if $\eps \in (0,1)$ is sufficiently small and $\Ao$ is selected as above, then at
least a $1-\eps$-fraction of the vertices of $\cK_C$ become infected if we consider a bootstrap percolation process on $CL(\mathbf{w})$
with activation threshold $r$ where the vertices in $[n] \setminus \cK_C$ are assumed to be ``frozen", that is, they never get infected.
\begin{lemma}[Proposition 3.7~\cite{ar:AmFount2012}] \label{lem:CKernelInfection}
There exists an $\eps_0 = \eps_0 (\beta, c_1, c_2) >0$ such that for any positive $\eps < \eps_0$
there exists $C=C( c_1, c_2, \beta, \eps, r) >0$ for which the following holds.
Assume that $\Ao$ is as above and consider a bootstrap percolation process on $CL(\bw)$ with activation threshold $r\geq 2$ and
the set $\Ao$ as the initial set, with the restriction that the vertices in $[n] \setminus \{ \cK_C \cup \Ao\}$ never become infected.
Then at least $(1- \eps) |\cK_C|$ vertices of $\cK_C$ become infected with probability $1-o(1)$.
\end{lemma}
Let $\E_{C,\eps, n}$ denote this
event and, if it is realised, we let $\cK_{C,\eps}$ denote a subset of $\lfloor (1-\eps)|\cK_{C}| \rfloor=:k$ vertices in
$\cK_C$ that become infected chosen in some particular way (for example, the $k$ lexicographically smallest vertices).
Hence, the following holds.
\begin{claim} \label{clm:LowerInclusion}
For any $C>0$ and any $\eps \in (0,1)$, there is a coupling such that if $\E_{C, \eps, n}$ is realised, then we have
$$  \Af (\cK_{C,\eps}) \subseteq \A_f.$$
\end{claim}
Let $\gamma \in F([0,\infty) )$ be such that $C_\gamma = C$, where $C = C (\eps)$ is as in Lemma~\ref{lem:CKernelInfection}.

Consider a set of vertices $\{v_1,\ldots, v_k\}$ which is disjoint from $[n]$.
We define a sequence $\tilde{\bW}^{(\ell, \gamma)-}$
on $\left(\cup_{i=1}^{p_\ell} \cC_i \right) \bigcup \{v_1,\ldots, v_k \}$ as follows. For every $j \in \cC_i$, with $i=1,\ldots, p_\ell$,
we have $\tilde{W}^{(\ell, \gamma)-}_j=W_j^{(\ell, \gamma)-}$, whereas for
every $j=1,\ldots, k$ we let $\tilde{W}^{(\ell, \gamma)-}_{v_j} = C_\gamma$.
We let $n_-$ be the number of vertices of the sequence $\tilde{\bW}^{(\ell,\gamma)-}$, that is, the size of
$\left(\cup_{i=1}^{p_\ell} \cC_i \right) \bigcup \{v_1,\ldots, v_k \}$. Since $k = (1-\eps) \gamma n (1+o(1))$, this satisfies
$n_- =( (1-\gamma) + \gamma(1-\eps) )n (1+o(1)) = (1-\gamma\eps) n(1+o(1))$.
Hence, for large $n$ we have $n_- < n$. We identify the vertices in $\{ v_1,\ldots, v_k \}$ with the lexicographically $k$ first
vertices in $\cC_\gamma$ and we denote both subsets by $\cC_{\gamma, k}$ (that is, both $\{ v_1,\ldots, v_k \}$ and the corresponding
subset of $\cC_\gamma$).
Setting $\tilde{W}_\gamma^- : = (1-\eps) \gamma C_\gamma $, the weight of these vertices is $n\tilde{W}_\gamma^- (1+o(1))$,
since each of them has weight equal to $C_\gamma$.

The weight sequence $\tilde{\bW}^{(\ell,\gamma)-}$ gives rise to a probability distribution which is the limiting probability distribution
of the weight of a uniformly chosen vertex from $[n_-]$. We let $\tilde{U}^{(\ell,\gamma)-}$ be a random variable which follows this
distribution and let $\tilde{W}_F^{(\ell,\gamma)-}$ denote a random variable which follows the $\tilde{U}^{(\ell,\gamma)-}$ size-biased
distribution. The definition of $\tilde{\bW}^{(\ell,\gamma)-}$ yields
$$ \PRO {\tilde{U}^{(\ell,\gamma)-} = W_i^-} = {\gamma_i \over 1- \gamma \eps}, \
\mbox{and} \ \PRO {\tilde{U}^{(\ell,\gamma)-} = C_\gamma} = {(1- \eps ) \gamma \over 1-\gamma \eps}. $$

As is Lemma~\ref{lem:size_biased}, one can show that $\tilde{\bW}^{(\ell,\gamma)-}$ is an $F$-convergent weight sequence
with error $\rho$, where $\rho=\rho (\gamma)$ is a function such that $\rho(\gamma) \downarrow 0$ as $\gamma \downarrow 0$.
We omit the proof.


Let $\hat{\Af} ( \cC_{\gamma,k} )$ be the final set of infected vertices in $CL (\bw )$
assuming that the initial set is $\cC_{\gamma,k}$ and moreover no vertices in
$\cC_{\gamma} \setminus \cC_{\gamma, k}$ ever become infected.
Hence, on the event $\mathcal{E}_{C_\gamma, \eps, n}$ we have
$$ |\hat{\Af} ( \cC_{\gamma,k} )| \leq_{st} |\Af ( \cK_{C_\gamma, \eps} )|.$$
But the assumption that no vertices in $\cC_{\gamma} \setminus \cC_{\gamma, k}$ ever become active
amounts to a bootstrap percolation process on $CL' (\tilde{\bW}^{(\ell,\gamma)-})$ with activation threshold equal to $r$.
Let $\tilde{\Af}  (S)$ denote the final set in this graph under the assumption that the initial set is $S \subseteq [n']$.
Since  $CL' (\tilde{\bW}^{(\ell,\gamma)-}) \subseteq CL  (\bw )$ on a certain coupling space we have
$$ |\tilde{\Af} ( \cC_{\gamma,k} )| \leq_{st} |\hat{\Af} (\cC_{\gamma,k} )|. $$
Therefore
$$|\tilde{\Af} (\cC_{\gamma,k} )| \leq_{st} | \Af ( \cK_{C_\gamma, \eps} )|. $$
This together with Claim~\ref{clm:LowerInclusion} imply the following stochastic lower bound on $|\Af|$.
\begin{claim} \label{clm:LowerDomination}
For any $\gamma, \eps \in (0,1)$, if $\E_{C_{\gamma},\eps, n}$ is realised, then
\begin{equation*}
|\tilde{\Af} ( \cC_{\gamma,k} ) | \leq_{st} |\Af|.
\end{equation*}
\end{claim}

We will now apply Theorem~\ref{thm:FinalInfectionI} to the random variables that bound $|\Af|$ in Claims~\ref{clm:UpperDomination}
and~\ref{clm:LowerDomination}.
Let $\hat{y}_\eps^+,\hat{y}$ be the smallest positive solutions of
$$y = (1-\eps)~\ex{ \psi_r \left(W_F^* y\right) } + \eps,$$
and
$$y = \ex{ \psi_r \left(W_F^* y\right) },$$
respectively.

For $\eps < \eps_0$ let $C$ be as in Lemma~\ref{lem:CKernelInfection} and let $\gamma < \gamma_2$
(cf. Theorem~\ref{thm:FinalInfectionI}) be such that $C = C_\gamma$.
Theorem~\ref{thm:FinalInfectionI} implies that for any $\delta \in (0,1)$
there exists an infinite set of natural numbers
$\mathcal{S}^1$ such that for every $\ell \in \mathcal{S}^1$ with probability $1-o(1)$
\begin{equation} \label{eq:UpperBoundI}
\begin{split}
{|\Af^+ (\Ao^{(\eps)} \cup \cC_\gamma^+ )| \over n} \leq (1+ \delta) ((1-\eps ) \ex {W_F \hat{y}_\eps^+} + \eps),
\end{split}
 \end{equation}
and an infinite set of natural numbers $\mathcal{S}^2$ such that for every $\ell \in \mathcal{S}^2$ with probability $1-o(1)$
\begin{equation} \label{eq:LowerBoundII}
{|\tilde{\Af} ( \cC_{\gamma, k} ) | \over n} \geq (1-\delta) \ex {W_F \hat{y}}
\end{equation}
Hence,  Claims~\ref{clm:UpperDomination} and~\ref{clm:LowerDomination} together with
(\ref{eq:UpperBoundI}) and (\ref{eq:LowerBoundII}) imply that a.a.s.
$$ {|\Af| \over n} \leq (1+ \delta) ((1-\eps ) \ex {W_F \hat{y}_\eps^+} + \eps), $$
and
$$ {|\Af| \over n} \geq (1- \delta) \ex {W_F \hat{y}}. $$
But $y_\eps^+ \rightarrow \hat{y}$ as $\eps \rightarrow 0$ and Theorem~\ref{thm:power_law} follows.

\section{Proof of Theorem~\ref{thm:FinalInfectionI}}

In this section we will give the proof of Theorem~\ref{thm:FinalInfectionI}. At the moment, our analysis does not depend on the
parameters $\ell, \gamma$ and, to simplify notation, we will drop the superscript
$(\ell,\gamma)$. For $j=0,\ldots, r-1$, we denote by $\cC_{i,j}$ the subset of $\cC_i$ which consists of those vertices
of $\cC_i$ which have $j$ infected neighbours. We also denote by $\cC_{i,r}$ the subset of $\cC_i$ containing all those vertices that are
infected, that is, they have \emph{at least} $r$ infected neighbours. 

We will determine the size of the final set of infected vertices exposing \emph{sequentially} the neighbours of each infected vertex and
keeping track of the number of infected neighbours an uninfected vertex has. In other words, we will be keeping track of the size
of the sets $\cC_{i,j}$. This method of exposure has also been applied in the analysis in~\cite{ar:JLTV10}. However, the inhomogeneity
in the present context bears additional difficulties as the evolutions of the sets $\cC_{i,j}$ are interdependent.

The \emph{sequential} exposure proceeds as follows. For $i=1,\ldots, p_\ell$ and $j=0,\ldots, r$, let $\cC_{i,j}(t)$ denote set
$\cC_{i,j}$ after the execution of the $t$th step.  Here $\cC_{i,j}(0)$ denotes the set $\cC_{i,j}$ before the beginning of the execution.
Furthermore, let $\cU (t)$ denote the set of infected \emph{unexposed} vertices after the execution of the $t$th step, with $\cU(0)$ denoting
the set of infected vertices before the beginning of the process.

At step $t\geq 1$, if $\cU(t-1)$ is non-empty,
\begin{enumerate}
\item[i.] choose a vertex $v$ uniformly at random from $\cU(t-1)$;
\item[ii.] expose the neighbours $v$ in the set $\bigcup_{i=1}^{p_\ell} \cup_{j=0}^{r-1} \cC_{i,j}(t-1)$;
\item[iii.] set $\cU(t) := \cU(t-1) \setminus \{ v\}$.
\end{enumerate}
The above set of steps is repeated for as long as the set $\cU$ is non-empty. The exposure of the neighbours of $v$ can be alternatively
thought of as a random assignment of a mark to each vertex of $\bigcup_{i=1}^{p_\ell} \cup_{j=0}^{r-1} \cC_{i,j}(t-1)$ independently
of every other vertex; if a vertex in $\cC_{i,j}(t-1)$ receives such a mark, then it is moved to $\cC_{i,j+1} (t)$. Hence, during the execution
of the $t$th step each vertex in $\cC_{i,j}(t-1)$ either remains a member of $\cC_{i,j}(t)$ or it is moved to $\cC_{i,j+1}(t)$.

\subsection{Conditional Expected Evolution}
Let $c_{i,j}$ denote the size of the set $\cC_{i,j}$ for all $i=1,\ldots, p_\ell$ and $j=0,\ldots, r-1$.
Our equations will also incorporate the size of $\cU$ at time $t-1$, which we denote by $u (t-1)$, as well as the total weight of
vertices $\cU$, which we denote by $w_{\cU} (t-1)$.
For these values of $i$ and $j$ we let $\V (t) = \left( u(t), w_{\cU} (t), ( c_{i,j}(t) )_{i,j}\right)$.
This vector determines the state of the process after step $t$.
We will now give the expected change of $c_{i,j}$ during the execution of step $t$, conditional on $\V (t-1)$. If step $t$ is to be
executed, it is necessary to have $u(t-1)>0$, which we will assume to be the case.
We begin with $c_{i,0}$, for $i=1,\ldots, p_\ell$, having
\begin{equation} \label{eq:ci0Expected}
\begin{split}
\ex {c_{i,0} (t) - c_{i,0} (t-1) \ | \ \V (t-1)} &= -c_{i,0} (t-1) \sum_{v \in \cU (t-1)} { W_i w_v \over W_{[n]}}~{1\over u (t-1)} \\
& = -c_{i,0} (t-1)~{ W_i \over W_{[n]}}~{w_{\cU} (t-1)\over u (t-1)}.
\end{split}
\end{equation}
The evolution of $c_{i,j}$ for $0< j < r$ involves a term that accounts for the ``losses" from the set $c_{i,j}$ as well as a term which
describes the expected ``gain" from the set $c_{i,j-1}$. For $i= 1,\ldots, p_\ell$ and $0< j < r$ we have
\begin{equation} \label{eq:cijExpected}
\begin{split}
 &\ex {  c_{i,j} (t) - c_{i,j} (t-1) \ | \ \V (t-1)} \\
= &  c_{i,j-1}(t-1) \sum_{v \in \cU (t-1)} { W_i w_v \over W_{[n]}}~{1\over u (t-1)}
- c_{i,j} (t-1) \sum_{v \in \cU (t-1)} { W_i w_v \over W_{[n]}}~{1\over u (t-1)} \\
=& (c_{i,j-1} (t-1) - c_{i,j} (t-1))~{ W_i \over W_{[n]}}~{w_{\cU} (t-1)\over u (t-1)}.
\end{split}
\end{equation}
Finally, we will need to describe the expected change in the size of $\cU$ during step $t$. In this case, \emph{one} vertex is removed
from $\cU (t -1)$, but additional vertices may be added from the sets $\cC_{i,r-1} (t-1)$. More specifically, we write
\begin{equation} \label{eq:uExpected}
\begin{split}
\ex {u (t) - u (t-1) \ | \ \V (t-1)} &= - 1 + \sum_{i=1}^{p_\ell} c_{i,r-1} (t-1)~\sum_{v \in \cU (t-1)} { W_i w_v \over W_{[n]}}~{1\over u (t-1)}
\\
& = -1 + {w_{\cU} (t-1)\over u (t-1)} \sum_{i=1}^{p_\ell} {W_i \over W_{[n]}}~c_{i,r-1} (t-1).
\end{split}
\end{equation}
Similarly, the expected change in the weight of $\cU$ during step $t$ is as follows:
\begin{equation} \label{eq:wuExpected}
\begin{split}
& \ex  {w_{\cU} (t) - w_{\cU} (t-1) \ | \ \V (t-1)} \\
= &  - {w_{\cU} (t-1)\over u (t-1)} + \sum_{i=1}^{p_\ell} W_i c_{i,r-1} (t-1)~\sum_{v \in \cU (t-1)} { W_i w_v \over W_{[n]}}~{1\over u (t-1)} \\
= & -{w_{\cU} (t-1)\over u (t-1)} + {w_{\cU} (t-1)\over u (t-1)} \sum_{i=1}^{p_\ell} {W_i^2 \over W_{[n]}}~c_{i,r-1} (t-1).
\end{split}
\end{equation}

\subsection{Continuous Approximation}

The above quantities will be approximated by the solution of a system of ordinary differential equations. We will consider
a collection of continuous differentiable functions $\gamma_{i,j}: [0,\infty) \rightarrow \mathbb{R}$, for all $i=1,\ldots, p_\ell$ and
$j=0,\ldots, r-1$, through which we will approximate the quantities $c_{i,j}$. To be more precise, $\gamma_{i,j}$ will be shown to be close
to $c_{i,j}/n$. Moreover, $u$ and $w_{\cU}$ will be approximated through the continuous differentiable functions
$\nu, \mu_{\cU} : [0,\infty) \rightarrow \mathbb{R}$
in a similar way. We will also use another continuous function $G: [0,\infty) \rightarrow \mathbb{R}$ which will approximate
the ratio $w_{\cU} / u$; note that this is the average weight of the set of infected unexposed vertices.

The system of differential equations that determine the functions $\gamma_{i,j}$ is as follows:
\begin{equation}\label{eq:DifferentialEquations}
\begin{split}
{d \gamma_{i,0} \over d\tau} &= - \gamma_{i,0} (\tau) {W_i \over d} G (\tau), \\
{d \gamma_{i,j} \over d\tau} &= \left(\gamma_{i,j-1} (\tau) - \gamma_{i,j} (\tau) \right) {W_i \over d} G (\tau), \quad 
\mbox{ $1 \le j \le r-1$}.
\end{split}
\end{equation}
The continuous counterparts of (\ref{eq:uExpected}) and (\ref{eq:wuExpected}) are
\begin{equation} \label{eq:uDifferentialEquation}
{d \nu \over d \tau} = -1 + G (\tau) \sum_{i=1}^{p_\ell} {W_i \over d}~\gamma_{i,r-1} (\tau),
\end{equation}
and
\begin{equation} \label{eq:wuDifferentialEquation}
{d \mu_{\cU} \over d\tau} = - G (\tau) + G (\tau) \sum_{i=1}^{p_\ell} {W_i^2 \over d}~\gamma_{i,r-1} (\tau).
\end{equation}
The initial conditions are 
\begin{equation} \label{eq:InitialConditions}
\begin{split}
\nu (0) &= p~(1-\gamma) + \gamma ',  \ \mbox{for $p \in [0,1)$ (recall that $p$ is the initial infection rate)}, \\
\mu_{\cU} (0) &=W_{\gamma}' +p \sum_{i=1}^{p_\ell} W_i \gamma_i , \\
\gamma_{i,0}(0) &= (1-p) \gamma_i , \\
\gamma_{i,j}(0) &= 0, \ \mbox{for $j=1,\ldots, r-1$}.
\end{split}
\end{equation}
In the following proposition, we will express the formal solution of the above system in terms of $\gamma_{i,0} (\tau)$.
\begin{proposition} \label{prop:Solution}
With $I(\tau) = \int_0^{\tau} G(s) ds$, we have
$$ \gamma_{i,0}(\tau) = \gamma_{i,0} (0) \exp \left( - {W_i  }  I(\tau ) / d\right) .$$
Moreover, for $1\le j \le r-1$ 
$$ \gamma_{i,j} (\tau) = {\gamma_{i,0}(\tau) \over j!} \log^j \left( {\gamma_{i,0} (0) \over \gamma_{i,0} (\tau)} \right). $$
\end{proposition}
\begin{proof}
The expression for $\gamma_{i,0} (\tau)$ can be obtained through separation of variables -- we omit the details.
The remaining expressions will be obtained by induction.
Let us consider the differential equation for $\gamma_{i,j}$, where $0< j < r$, assuming that we have derived the expression for
$\gamma_{i,j-1}$. This differential  equation is a first order ordinary differential equation of the form
$y'(\tau) = a(\tau )y(\tau ) + b(\tau)$ with initial condition $y(0)=0$. Its general solution is equal to
$$ y(\tau) = \exp \left( \int_0^{\tau} a(s) ds \right) \cdot \int_0^{\tau} b(s) \exp \left( - \int_0^{s} a(\rho) d\rho \right) ds. $$
Here, we have
$$ a(\tau) = - {W_i \over d} G(\tau ), \ b(\tau) = \gamma_{i,j-1}(\tau ) {W_i \over d} G(\tau ) =
{W_i \over d}~{\gamma_{i,0}(\tau) \over (j-1)!}~\log^{j-1} \left( {\gamma_{i,0} (0) \over \gamma_{i,0} (\tau)} \right)  G(\tau ),$$
by the induction hypothesis. Thereby and using the expression for $\gamma_{i,0}$ we obtain
\begin{equation} \label{eq:Factor1} \exp \left( \int_0^{s} a(\rho) d\rho \right)= {\gamma_{i,0} (s ) \over \gamma_{i,0} (0)}.
\end{equation}
Hence
\begin{equation} \label{eq:Factor2}
\begin{split}
\int_0^{\tau}  b(s) \exp & \left(  - \int_0^{s} a(\rho) d\rho \right) ds = \\
& {W_i \over d (j-1)!}~\int_0^{\tau} \gamma_{i,0} (s) \log^{j-1} \left( {\gamma_{i,0} (0) \over \gamma_{i,0} (s)} \right)~G(s)~
{\gamma_{i,0} (0) \over \gamma_{i,0} (s)} ds \\
& = \gamma_{i,0}(0)~{W_i \over d (j-1)!}~\int_0^{\tau} \gamma_{i,0} (s)  \log^{j-1} \left( {\gamma_{i,0} (0) \over \gamma_{i,0} (s)} \right)~{G(s) \over \gamma_{i,0}(s)} ds \\
& = - {\gamma_{i,0} (0) \over (j-1)!}~\int_0^{\tau} {1\over \gamma_{i,0}(s)}
\log^{j-1} \left( {\gamma_{i,0} (0) \over \gamma_{i,0} (s)} \right)
\left(-  \gamma_{i,0}(s) {W_i \over d}G(s)  \right) ds \\
& \stackrel{(\ref{eq:DifferentialEquations})}{=}
- {1 \over (j-1)!}~\int_0^{\tau} {\gamma_{i,0} (0)\over \gamma_{i,0}(s)}
\log^{j-1} \left( {\gamma_{i,0} (0) \over \gamma_{i,0} (s)} \right)
\left( {d \gamma_{i,0}\over ds} \right) ds \\
& = - {\gamma_{i,0} (0) \over (j-1)!}~\int_0^{\tau} {\gamma_{i,0} (0)\over \gamma_{i,0}(s)}
\log^{j-1} \left( {\gamma_{i,0} (0) \over \gamma_{i,0} (s)} \right) d \left( {\gamma_{i,0}\over \gamma_{i,0}(0)} \right) \\
& \stackrel{\left( x= \gamma_{i,0}/\gamma_{i,0}(0) \right)}{=}
- {\gamma_{i,0} (0) \over (j-1)!}~\int_{1}^{\gamma_{i,0}(\tau)/\gamma_{i,0}(0)} {1\over x} \log^{j-1} \left({1\over x} \right) dx \\
&= (-1)^{j-1} {\gamma_{i,0} (0) \over (j-1)!}~\int_{\gamma_{i,0}(\tau)/\gamma_{i,0}(0)}^1 {\log^{j-1} (x) \over x} dx.
\end{split}
\end{equation}
For $j=1$, the last integral equals $\log (\gamma_{i,0}(0)/\gamma_{i,0}(\tau))$.
For $j\geq 2$, it can be calculated using integration by parts.
\begin{equation*}
\int  {\log^{j-1} (x) \over x} dx = \int \left( \log (x) \right)' \log^{j-1}(x) dx = \log^j (x) - (j-1) \int {\log^{j-1}(x)\over x}dx,
\end{equation*}
which yields
\begin{equation*}
\int {\log^{j-1}(x)\over x}dx = {\log^j (x) \over j}.
\end{equation*}
Thereby, the last integral in (\ref{eq:Factor2}) is
\begin{equation*}
\int_{\gamma_{i,0}(\tau)/\gamma_{i,0}(0)}^1 {\log^{j-1} (x) \over x} dx = - {1\over j}~ \log^j \left(
{\gamma_{i,0}(\tau) \over \gamma_{i,0}(0)} \right) = {(-1)^{j+1}\over j}~\log^j \left( {\gamma_{i,0}(0) \over \gamma_{i,0}(\tau)} \right).
\end{equation*}
Substituting this into (\ref{eq:Factor2}) we obtain:
\begin{equation} \label{eq:Factor2Final}
\int_0^{\tau}  b(s) \exp  \left(  - \int_0^{s} a(\rho) d\rho \right) ds =
{\gamma_{i,0} (0) \over j!}~\log^j \left( {\gamma_{i,0}(0) \over \gamma_{i,0}(\tau)} \right).
\end{equation}
Combining (\ref{eq:Factor1}) and (\ref{eq:Factor2Final}), we have
$$ \gamma_{i,j} (\tau) = {\gamma_{i,0} (\tau ) \over j!}~\log^j \left( {\gamma_{i,0}(0) \over \gamma_{i,0}(\tau)} \right).$$
\end{proof}
In the sequel we will use the expressions for $\gamma_{i,r-1}$, where $1 \le i \le p_\ell$, and  integrate 
(\ref{eq:uDifferentialEquation}), (\ref{eq:wuDifferentialEquation}) in order to deduce the expressions for $\nu$ and $\mu_{\cU}$.
\begin{proposition} \label{eq:uSolution}
We have
\begin{equation*}
\nu (\tau) = p~(1-\gamma ) +\gamma' - \tau + (1-p) \sum_{i=1}^{p_\ell}  \gamma_i \PRO {\Po \left( {W_i \over d} I(\tau) \right) \geq r}
 \end{equation*}
 and
\begin{equation*}
\mu_{\cU} (\tau ) = W_{\gamma}' +p \sum_{i=1}^{p_\ell} W_i \gamma_i - I(\tau )
+ (1-p) \sum_{i=1}^{p_\ell}  W_i \gamma_i \PRO {\Po \left( {W_i \over d} I(\tau) \right) \geq r}.
\end{equation*}
\end{proposition}
\begin{proof}
Applying Proposition~\ref{prop:Solution} to \eqref{eq:uDifferentialEquation} yields
\begin{equation*}
{d \nu \over d \tau} = -1 + G (\tau) \sum_{i=1}^{p_{\ell}} {W_i \over d}
~{\gamma_{i,0}(\tau) \over (r-1)!} \log^{r-1} \left( {\gamma_{i,0} (0) \over \gamma_{i,0} (\tau)} \right).
\end{equation*}
By integrating this expression we obtain
\begin{equation} \label{eq:uEquation}
\begin{split}
\nu (\tau ) &= \nu (0) - \tau + {1\over (r-1)!}\sum_{i=1}^{p_\ell} \int_0^{\tau} {W_i \over d}
~\gamma_{i,0}(s) G(s)  \log^{r-1} \left( {\gamma_{i,0} (0) \over \gamma_{i,0} (s)} \right) ds \\
& \stackrel{(\ref{eq:DifferentialEquations})}{=}
\nu (0) - \tau - {1\over (r-1)!}\sum_{i=1}^{p_\ell} \int_0^{\tau}
\left({d \gamma_{i,0} \over ds}\right) \, \log^{r-1} \left( {\gamma_{i,0} (0) \over \gamma_{i,0} (s)} \right)  ds \\
&= \nu (0) - \tau - {1 \over (r-1)!}\sum_{i=1}^{p_\ell} \gamma_{i,0}(0)  \int_1^{\gamma_{i,0}(\tau)/\gamma_{i,0}(0)}
\log^{r-1} \left( {1\over x}\right) dx.
\end{split}
\end{equation}
We calculate the last integral substituting $y$ for $1/x$ and using integration by parts. We have
\begin{equation*}
\begin{split}
\int \log^{r-1} \left( {1\over x} \right) dx &= - \int {\log^{r-1} (y) \over y^2 } dy = \int \left({1\over y} \right)' \log^{r-1} (y) dy \\
& = { \log^{r-1} (y) \over y} - (r-1) \int {\log^{r-2} (y) \over y^2} dy.
\end{split}
\end{equation*}
As $\int {1\over y^2} dy = - {1\over y}$, dividing and multiplying by $(r-1)!$, we obtain
$$ \int \log^{r-1} \left( {1\over x} \right) dx = {(r-1)! \over y}~ \sum_{i=0}^{r-1} {\log^i (y) \over i!},$$
where $y=1/x$. Thereby, for all $i=1,\ldots, p_\ell$ we have
\begin{equation*}
\begin{split}
\int_1^{\gamma_{i,0}(\tau)/\gamma_{i,0}(0)} \log^{r-1} \left( {1\over x}\right) dx =
(r-1)! \left( {\gamma_{i,0}(\tau ) \over \gamma_{i,0}(0 )}~\sum_{i=0}^{r-1}{1\over i!}
\log^i \left({\gamma_{i,0}(0) \over \gamma_{i,0}(\tau)} \right) - 1\right).
\end{split}
\end{equation*}
Substituting the above into (\ref{eq:uEquation}) we obtain
$$
\nu (\tau) = \nu (0) - \tau  + \sum_{i=1}^{p_\ell} \gamma_{i,0}(0)~ \left(1- {\gamma_{i,0}(\tau ) \over \gamma_{i,0}(0 )}~\sum_{j=0}^{r-1}{1\over j!} \log^j \left({\gamma_{i,0}(0) \over \gamma_{i,0}(\tau)} \right) \right).
$$
Observe now that the expression in brackets is equal to the probability that a Poisson distributed random variable
with parameter $\log \left( \gamma_{i,0}( 0) / \gamma_{i,0}(\tau) \right)$ is at least $r$.
But by Proposition~\ref{prop:Solution}, we have
$$\log \left( {\gamma_{i,0}( 0) \over \gamma_{i,0}(\tau)} \right) = {W_i \over d} I(\tau). $$
Also, recall that by (\ref{eq:InitialConditions}) $\gamma_{i,0} (0) = (1-p) \gamma_i$, for each $i=1,\ldots, p_\ell$, and
$\nu (0) = p~( 1- \gamma )+\gamma'$.
Hence
\begin{equation*}
\nu (\tau) = p~(1-\gamma ) + \gamma' - \tau + (1-p) \sum_{i=1}^\ell  \gamma_i \PRO {\Po \left( {W_i \over d} I(\tau) \right) \geq r}.
 \end{equation*}
 The expression of $\mu_{\cU}$ is obtained along the same lines and we omit its proof.
\end{proof}

\subsection{Wormald's Theorem}
We summarize here the method introduced by Wormald in~\cite{Worm95,Worm99} for the analysis of a discrete random process by using differential equations. Recall that a function $f(u_1, . . . , u_{b+1})$ satisfies a Lipschitz condition in a domain $D \subseteq \RR^{b+1}$ if there is a constant $L > 0$
such that
$$|f(u_1, . . . , u_{b+1})-f(v_1, . . . , v_{b+1})| \leq L \max _{1\leq i \leq b+1}|u_i-v_i|$$
for all $(u_1, ..., u_{b+1}), (v_1, ...., v_{b+1}) \in D$. For the random process $(Y_1(t), ..., Y_b(t))\in \mathbb{R}^b$, 
the~\emph{stopping time} $T_D(Y_1, ..., Y_b)$ is defined to be the minimum $t$ such that $$(t/n; Y_1(t)/n, ..., Y_b(t)/n) \notin D.$$
This is written as $T_D$ when $Y_1, ..., Y_b$ are understood from the context.


\begin{theorem}[\cite{Worm99}]
\label{thm-eqdif1}
Let $b,n \in \mathbb{N}$. For $1 \leq j \leq b$, suppose that $Y^{(n)}_j(t)$ is a sequence of real-valued random variables such that $0 \leq Y^{(n)}_j \leq C n$ for some constant $C > 0$. Let $H_t$ be the history up to time $t$, i.e., the sequence $\{Y^{(n)}_j(k), \ 0 \leq j \leq b, \ 0 \leq k \leq t\}$. Suppose also that for some bounded connected open set $D \subseteq \RR^{b+1}$ containing the intersection of $\{(t,z_1,...,z_b): t\geq 0\}$ with some neighborhood of
$$\left\{(0,z_1,...,z_b):\PP(Y^{(n)}_j(0)=z_j n, 1\leq j\leq b) \neq 0 \mbox{ for some n}\right\},$$
the following three conditions are satisfied:
\begin{enumerate}
  \item {\rm(Boundedness).} For some functions $\omega=\omega(n)$ and $\lambda=\lambda(n)$ with $\lambda^4\log n < \omega < n^{2/3}/\lambda$ and $\lambda\to \infty$ as $n\to \infty$, for all $l\leq b$ and uniformly for all $t<T_D$,
  $$\PP \left(|Y^{(n)}_l(t+1)-Y^{(n)}_l(t)|>\frac{\sqrt{\omega}}{\lambda^2\sqrt{\log n}}\mid H_t\right) = o(n^{-3});$$

  \item {\rm(Trend).} For all $l \leq b$ and uniformly over all $t<T_D$,
  $$\EE[Y^{(n)}_l(t+1)-Y^{(n)}_l(t)|H_{t}] = f_l(t/n,Y_1^{(n)}(t)/n,...,Y_b^{(n)}(t)/n) + o(1) ;$$

  \item {\rm(Lipschitz).}  For each $l$ the function $f_l$ is continuous and satisfies a Lipschitz condition on $D$ with all Lipschitz constants uniformly bounded.
\end{enumerate}
Then the following hold.
\begin{itemize}
  \item[$(a)$] For $(0,\hat{z}_1,...,\hat{z}_b) \in D$, the system of differential equations
  $$\frac{dz_l}{ds}=f_l(s,z_1,...,z_l), \ \ l=1,...,b ,$$ has a unique solution in $D$, $z_l:\RR \rightarrow \RR$ for $l=1,\dots,b$, which passes through $z_l(0)=\hat{z}_l,$ $l=1,\dots,b,$, and which extends to points arbitrarily close to the boundary of $D$.

\item[$(b)$] We have
$$Y_l^{(n)}(t)=n z_l(t/n) + o_p(n)$$
uniformly for  $0 \leq t \leq \min\{\sigma n,T_{D}\}$ and for each $l$. Here $z_l$ is the solution in (a) with $\hat{z}_l = Y^{(n)}_l(0)/n$, and $\sigma = \sigma_D(n)$ is the supremum of those $s$ to which the solution can be extended.

\end{itemize}
\end{theorem}

\subsection{Proof of Theorem~\ref{thm:FinalInfectionI}}
We will apply Theorem~\ref{thm-eqdif1}
to show that the trajectory of $\{u(t), w_{\cU}(t), ( c_{i,j}(t) )_{1\le i \le p_\ell, 0 \le j \le r-1}\}$ throughout the algorithm is a.a.s. close to the solution
of the deterministic equations suggested by these equations, i.e., $\{\nu, \mu_{\cU}, ( \gamma_{i,j} )_{i=1,\dots,p_\ell, j=0,\dots,r-1}\}$.


We set $b=r p_\ell + 2$. For $\epsilon >0$, we define
\begin{align*}
D_{\epsilon} = \{ (\tau, \nu, \mu_{\cU}, (\gamma_{i,j})_{i,j}) \in \RR^{b+1} \mid & -\epsilon<\tau<1, \ 0 < \frac{\mu_{\cU}}{\nu} < 2 C_\gamma, \ -\epsilon< \gamma_{i,j}< \gamma_i+\epsilon, \\ \epsilon < \mu_{\cU} < W_{\gamma}' + \sum_{i=1}^{p_\ell} W_i \gamma_i \},
\end{align*}
We now apply the last part ($b$) of Theorem~\ref{thm-eqdif1}. Note that Boundedness and Trend hypotheses are verified for
$t < T_{D_{\epsilon}}$.
More specifically, the Boundedness hypothesis follows since the changes in the quantities $u(t), w_{\cU}(t), c_{i,j}(t)$ are bounded by a
constant multiple of the maximum degree of the random graph. But since the maximum weight is bounded, we may choose, for example,
$\lambda = n^{1/8}$ and $\omega = n^{25/48}$, and show that the maximum
degree is bounded by $\sqrt{\omega} / ( \lambda^2 \log n)= n^{1/96}/\log n$ with probability $1-o(n^{-3})$.
The Trend hypothesis is verified by (\ref{eq:ci0Expected})--(\ref{eq:wuExpected}).
By the conditions that $0 < \frac{\mu_{\cU}}{\nu}<2C_{\gamma}$ and $\mu_{\cU} > \epsilon$, the Lipschitz condition is also verified.
Hence, for $0 \leq t \leq \min\{\sigma_{D} n,T_{D_{\epsilon}}\}$, we have
\begin{eqnarray} \label{eq:solutions}
u(t) &=& n \nu(t/n) + o_p(n), \nonumber \\
w_{\cU}(t) &=& n \mu_{\cU}(t/n) + o_p(n),\\
c_{i,j}(t) &=& n \gamma_{i,j}(t/n) + o_p(n), \ \text{for all} \  i=1,\dots,p_\ell, \  j=0,\dots,r-1. 
\nonumber
\end{eqnarray}
\noindent
This gives us the convergence up to the point where the solution leaves $D_{\epsilon}$. 
Observe that the definition of the domain $D_{\epsilon}$ together with the fact that the maximum weight is 
bounded by $2C_\gamma$ imply that at round $T_{D_\epsilon}$ we have 
$w_{\cU}(T_{D_\epsilon})/n < \epsilon$, but $w_{\cU}(T_{D_\epsilon} -1)/n \geq \epsilon$.   

Let $\Af (T_{D_\epsilon})$ be the set of infected vertices that have been exposed up to time $T_{D_{\epsilon}}$. 
Observe that $|\Af (T_{D_\epsilon})| =T_{D_{\epsilon}}$ as exactly one vertex is removed at each step. 
Also, as we noted above 
$ w_{\cU} (T_{D_{\epsilon}})/n < \epsilon$, but $ w_{\cU} (T_{D_{\epsilon}}-1)/n \geq \epsilon$. 
Since the maximum degree is $o_p(n)$ and the weights are bounded, a.a.s. we have 
$$\epsilon \leq  w_{\cU} (T_{D_{\epsilon}}-1)/n \leq 1.5\epsilon. $$
Hence, by (\ref{eq:solutions}) a.a.s.
\begin{equation} \label{eq:mu_small} 
\mu_{\cU} \left( \frac{T_{D_{\epsilon}}-1 }{n} \right) < 2 \epsilon.
\end{equation}
Also, as the maximum weight is bounded by $2 C_\gamma$, the bound on $w_{\cU}$ implies that  
\begin{equation} \label{eq:u_stop}  
u\left(T_{D_{\epsilon}}-1\right) /n \leq \frac{1.5 \epsilon}{2C_\gamma}. 
\end{equation}
Therefore, (\ref{eq:solutions}) again implies that a.a.s.
$$ \nu \left(\frac{T_{D_{\epsilon}}-1}{n} \right) \leq \frac{\epsilon}{C_\gamma}. $$
Let 
$$\alpha (y ) := p~(1-\gamma ) +\gamma'  +   (1-p) \sum_{i=1}^{p_\ell}  \gamma_i  \psi_r \left( W_i y \right). $$
The first part of Proposition~\ref{eq:uSolution} implies that 
\begin{equation} \label{eq:StopTime}
\left| \frac{T_{D_{\epsilon}}-1}{n} -  
\alpha \left(\frac{1}{d} I\left( \frac{T_{D_{\epsilon}}-1}{n}\right)  \right)
\right| \leq \frac{\epsilon}{C_\gamma}. 
\end{equation}
Let $\hat{\tau}^{(\ell, \gamma)}$ denote the minimum $\tau >0$ such that $\mu_{\cU} (\tau) = 0$.
By Lemma~\ref{lem:stableSol} below there exists $\gamma_2>0$ with the property that for any $\gamma < \gamma_2$ and
any $\delta \in (0,1)$ there exists
an infinite set of positive integers $\mathcal{S}$ such that when $\ell \in \mathcal{S}$, it holds that
\begin{equation}\label{muDerivative}\mu_{\cU}'( \hat{\tau}^{(\ell, \gamma)}) <0, \end{equation}
and 
\begin{equation} \label{eq:fin_approx}
\left|\alpha (\hat{y}_{\ell,\gamma }) - \left(p + (1-p)\mathbb{E}(\psi_r (W_F \hat{y})) \right)  \right| < \delta,
\end{equation}
where $\hat{y}_{\ell,\gamma}$ is the smallest positive root of
$$ y = {W_\gamma' \over d} + p {1\over d} \sum_{i=1}^{p_\ell} W_i \gamma_i
+ (1-p) \sum_{i=1}^{p_\ell} {W_i \gamma_i \over d} \psi_r (W_i y) .$$
Its existence is implied by the continuity of $I(\tau)$ and $\alpha (y)$. 
By (\ref{eq:mu_small}), the continuity of the function $\mu_{\cU}$ and its monotonicity around $\hat{\tau}^{(\ell,\gamma)}$ we deduce
that there exists
$\delta_1 = \delta_1 (\epsilon)>0$ such that, for $n$ large enough, 
$$\hat{\tau}^{(\ell,\gamma)} - \delta_1 < \frac{T_{D_{\epsilon}}-1}{n} \leq \hat{\tau}^{(\ell,\gamma)}. $$
The continuity of $I$ and $\alpha$ implies that there exists an increasing function $f:\mathbb{R}^+ \rightarrow \mathbb{R}^+$, 
such that $f(x) \downarrow 0$ as $x \downarrow 0$ and 
\begin{equation} \label{eq:alpha_approx}
\left| \alpha \left( \frac{1}{d} I\left( \hat{\tau}^{(\ell,\gamma)} \right) \right)  - 
\alpha \left( \frac{1}{d}I\left( \frac{T_{D_{\epsilon}}-1}{n} \right) \right) \right| < f (\delta_1) . 
\end{equation}

Let us set $x=x (\tau) = I(\tau) /d$. Since $\mu_{\cU} (\hat{\tau}^{(\ell,\gamma )}) = 0$, this implies that
\begin{equation*}
\begin{split}
{ I(\hat{\tau}^{(\ell,\gamma )}) \over d} &= {W_{\gamma}' \over d} +p~{1\over d}\sum_{i=1}^{p_\ell} W_i \gamma_i
+ (1-p) \sum_{i=1}^{p_\ell}  {W_i \gamma_i \over d} \PRO {\Po \left( {W_i \over d} I( \hat{\tau}^{(\ell,\gamma )}) \right) \geq r} \\
& = {W_{\gamma}' \over d} +p~{1\over d}\sum_{i=1}^{p_\ell} W_i \gamma_i
+ (1-p)  \sum_{i=1}^{p_\ell}  {W_i \gamma_i \over d} \psi_r  \left( {W_i \over d} I( \hat{\tau}^{(\ell,\gamma )}) \right),
\end{split}
\end{equation*}
whereby $\hat{y}_{\ell, \gamma} = x (\hat{\tau}^{(\ell,\gamma )})$. 
Thus the triangle inequality together with (\ref{eq:StopTime}), (\ref{eq:fin_approx}) and (\ref{eq:alpha_approx}) imply that for any $\gamma < \gamma_2$, any
$\delta \in (0,1)$ and any $\ell \in \mathcal{S}$ a.a.s.
\begin{equation*}
\begin{split}
\left| n^{-1}{|\Af (T_{D_\epsilon})|} - \alpha (\hat{y}) \right| < \frac{\epsilon}{C_\gamma} + 
\delta + f(\delta_1 ) + \frac{1}{n}.
\end{split}
\end{equation*}
Recall that $f(\delta_1)$ can become arbitrarily small if we make $\epsilon$ small enough. 
Therefore, the right-hand side of the above can become as small as we please. 
Since $u(T_{D_\epsilon}) \leq \frac{1.5\epsilon}{2C_\gamma}n$, 
the proof of Theorem~\ref{thm:FinalInfectionI} will be complete, if we show that the process will finish soon after 
$T_{D_\epsilon}$. 

More specifically, we will show that with high probability only a small fraction of vertices are added after $T_{D_\epsilon}$. 
From now on, we start exposing the edges incident to all vertices of $\cU(t)$ \emph{simultaneously}. 
Hence, we change the time scaling. Informally, 
each round is a generation of a multi-type branching process which is sub-critical. 
First, let us observe that the continuity of $\mu_{\cU}'$ together with (\ref{muDerivative}) and (\ref{eq:wuDifferentialEquation}) imply 
that there exists $\kappa_0<1$ such that for all $n$ sufficiently large we have
\begin{equation}\label{eq:eigenvalue} \sum_{i=1}^{p_\ell} \frac{W_i^2}{W_{[n]}} 
c_{i, r-1}\left( \frac{T_{D_{\epsilon}}-1}{n} \right) < \kappa_0< 1.
\end{equation}
We will stochastically bound from above the evolution of the process by a subcritical multi-type branching process in which 
the above expression dominates the principal eigenvalue of the expected progeny matrix. 
In fact, we will not keep track of the actual size of $\cU$ but of a functional which is well-known in the theory of 
multi-type branching processes to give rise to a martingale. However, in our context we will not have exactly the martingale property 
but only approximately. Let us proceed with the details of this argument. 

Let $\cU_i (t)$ denote the subset of $\cU (t)$ which consists of those vertices that have weight $W_i$ and let $u_i(t):= |\cU_i(t)|$ -- we 
say that these vertices are of type $i$. 
Let $\overline{u}_t= [u_1(t),\ldots, u_{p_\ell}(t)]^T$ be the vector whose co-ordinates are the sizes of the sets $\cU_i (t)$. 
A vertex $v \in \cU_j(t)$ can ``give birth" to vertices of type $i$ (i.e., of weight $W_i$). These may be vertices from the 
set $\mathcal{C}_{i,r-1}$ or from any one of the sets $\mathcal{C}_{i,r-k}$, for $k\geq 2$. 
If $v$ becomes adjacent to a vertex in $\mathcal{C}_{i,r-1}$, then this becomes infected and we say that it is a child of $v$. 
Similarly, we say that a vertex in $\mathcal{C}_{i,r-k}$, for $k\geq 2$, becomes a child of $v$, if it is adjacent to $v$ and 
to \emph{some other} vertex in $\cU(t)$. In that sense, a vertex may be a child of more than one vertices in $\cU(t)$. 
In this case, we assume that the vertex is \emph{born twice} and it is double-counted in $\cU(t+1)$. 
In fact, the former case is much more likely than the latter. 
The expected number of those children that are born out of $\mathcal{C}_{i,r-1}$
is bounded by $\frac{W_j W_i}{W_{[n]}} c_{i,r-1} (t)< \frac{W_jW_i}{d}(\gamma_{i,r-1}(t) + \delta)$, for 
any $\delta >0$, assuming that $c_{i,r-1} (t)$ is concentrated around $\gamma_{i,r-1} (t) n$. Indeed, this is the case 
for $t=(T_{D_{\epsilon}}-1)/n$, which will take as the starting time of our analysis. 
The expected number of the vertices of type $i$ that originate from $\mathcal{C}_{i,r-k}$, for $k\geq 2$, is 
bounded by $c_{i,r-k}\frac{W_j W_i}{W_{[n]}}  \left( |\cU(t)|   (2 C_\gamma)^2/W_{[n]} \right)$. 
This is the case as the factor $|\cU(t)|   (2 C_\gamma)^2/W_{[n]}$ bounds from above the probability that a given 
vertex in $\mathcal{C}_{i,r-k}$ is adjacent to some other vertex in $\cU(t)$. 

We set $t_0:=(T_{D_{\epsilon}}-1)/n$.   

Now, if we let $A_t$ be the $p_\ell \times p_\ell$ matrix whose $ij$ entry is the expected number of children of type $j$ that 
a vertex of type $i$ has, then 
$\Ex{ \overline{u}_{t+1} | \mathcal{H}_t}  = \overline{u}_t^{T} A_t$, where $\mathcal{H}_t$ is the sub-$\sigma$-algebra 
which is generated by the history of the process up to round $t$. 
One can view the matrix $A_t$ as the expected progeny matrix of a multi-type branching process, where the expected number of 
children of type $j$ that a vertex of type $i$ gives birth to is at most 
$$A_t[i,j]:=\frac{W_iW_j}{W_{[n]}}a_j(t), \ \mbox{where} \ a_j(t):=c_{j,r-1}(t) + 
u(t) \frac{4 C_\gamma^2}{W_{[n]}}\sum_{k=2}^{r} c_{j,r-k} (t). $$ 
Throughout this section, we will be working with this upper bound, which comes from a stochastic upper bound on the process. 
It is not hard to see that the vector $[W_1,\ldots, W_{p_\ell}]^T$ is a right eigenvector of $A_t$, with 
$$\sum_{i=1}^{p_\ell} \frac{W_i^2}{W_{[n]}} a_i(t)=:\rho_t$$ 
being the corresponding eigenvalue. In fact, this is the unique 
positive eigenvalue of $A_t$. Assuming that $c_{j,r-1}(t)$ does not decrease (which we can, taking a stochastic upper bound), 
we have $\rho_t \geq \rho_{t_0}$, for $t >t_0$.  

For $t=t_0$, it is not hard to see that $\rho_{t_0}$ is less than and bounded away from 1, if we choose $\epsilon$ small enough. 
Indeed, by (\ref{eq:u_stop}) 
$$ u(t_0) \frac{4 C_\gamma^2}{W_{[n]}}\sum_{k=2}^{r} c_{j,r-k} (t_0)  \leq 
u(t_0) \frac{4 C_\gamma^2}{W_{[n]}} n < u (t_0) \frac{5 C_\gamma^2}{d n} n \stackrel{(\ref{eq:u_stop})}{\leq }  
\epsilon {15 C_\gamma \over 2d}  n.$$
Hence, together with (\ref{eq:eigenvalue}) we deduce that if $\epsilon$ is small enough, then $\rho_{t_0}$ is smaller than 1 and, in fact, 
it is bounded away from 1. 

Let $\lambda_i: = W_i /\sum_j W_j$ and set $\xi:=[\lambda_1,\ldots, \lambda_{p_\ell}]^T$. Clearly, this is also a right eigenvector
of $A_t$. Consider now the random variable $Z_t = (\xi, \overline{u}_t)$, where $(\cdot ,\cdot)$ is the usual \emph{dot} product. 
Therefore, 
$$ \Ex {Z_{t+1} | \mathcal{H}_t} \leq \rho_t Z_t. $$
\begin{claim}\label{clm:conc_proc}
With (conditional) probability $1-o(n^{-1})$ we have 
$$ Z_{t+1} \leq \rho_t Z_t + Z_t^{1/2}\log^2 n.$$ 
\end{claim}
\begin{proof}[Proof of Claim~\ref{clm:conc_proc}] 
Note that $Z_{t+1}$ is a weighted sum of Bernoulli random variables, where the weights are bounded. 
More specifically, $Z_{t+1} = \sum_{j=1}^{p_\ell} \lambda_{j}\sum_{k=1}^r \sum_{v \in \mathcal{C}_{j,r-k}(t)} 
\mathbf{1}_{d_{\cU(t)}(v) \geq k}$. We will appeal to Talagrand's inequality (see for example Theorem 2.29 in~\cite{JLR}). 
Firstly, note that $Z_{t+1}$ is a function of independent Bernoulli random variables, which correspond to the (potential) edges that 
are incident to $\cU(t)$. If we change any one of them, then $Z_{t+1}$ will change accordingly by at most 1 (as all the 
$\lambda_j$s are at most 1). 
Furthermore, if $Z_{t+1}\geq x$, for some $x\geq 0$, then there are at most $rx$ edges whose presence witnesses this fact. 
Hence, we can apply Theorem 2.29 from~\cite{JLR} taking $\psi (x) =rx$, with $m(Z_{t+1})$ being the median of $Z_{t+1}$; 
Talagrand's inequality yields
\begin{equation}\label{eq:Tala} \Pro {Z_{t+1} \geq m(Z_{t+1}) + \frac{1}{2}Z_t^{1/2}\log^2 n}  \leq 2 e^{-\frac{Z_t \log^4 n}{4r\left(m(Z_{t+1}) +Z_t^{1/2}\log^2 n\right)}}.
\end{equation}
Since $\psi (x)$ is proportional to $x$ and $Z_{t+1}$ takes only non-negative integer values, (using an argument similar to that 
on pages 41--42 in~\cite{JLR}) it follows that 
$$ |\Ex{Z_{t+1}} - m(Z_{t+1}) | = O(\Ex {Z_{t+1}}^{1/2}). $$
Hence, for $n$ large enough
\begin{equation*}  
\begin{split}
&\Pro {Z_{t+1} \geq \Ex {Z_{t+1}} + Z_t^{1/2}\log^2 n} \leq \\
&\Pro {Z_{t+1} \geq m(Z_{t+1}) - O(\Ex {Z_{t+1}}^{1/2}) + Z_t^{1/2}\log^2 n}  \leq \Pro {Z_{t+1} \geq m(Z_{t+1}) + \frac{1}{2}Z_t^{1/2}\log^2 n}. 
\end{split}
\end{equation*}
So by (\ref{eq:Tala}) we conclude (using that $m(Z_{t+1}) \leq 2 \Ex {Z_{t+1}}\leq 2 \rho_t Z_{t}$) that 
\begin{equation*}
\begin{split}
&\Pro {Z_{t+1} \geq \Ex {Z_{t+1}} + Z_t^{1/2}\log^2 n} \leq \\
&\hspace{3cm} 2 e^{-\frac{Z_t \log^2 n}{4r\left(m(Z_{t+1}) +Z_t^{1/2}\log^2 n
 \right)}} 
\leq 2 e^{-\frac{Z_t \log^4 n}{4r \left( 2\rho_t Z_t +Z_t^{1/2}\log^2 n \right)}} = e^{- \Omega (\log^2 n)}. 
\end{split}
\end{equation*}
\end{proof}
We denote the above event by $\mathcal{E}_t$. 
Let $T = \min\{ t\geq t_0 \ : \  Z_t \leq n^{1/2}\}$ and let $t_0 \leq t <T$.
On $\mathcal{E}_t$ we have
\begin{equation} \label{eq:Z_tineq}
Z_{t+1} \leq \rho_t Z_t \left( 1 + \frac{Z_t^{1/2} \log^2 n}{\rho_t Z_t} \right) \stackrel{\rho_t \geq \rho_{t_0}}{\leq} \rho_t Z_t 
\left( 1 + \frac{\log^2 n}{\rho_{t_0} n^{1/4}} \right). 
\end{equation}
Thus, on $\cap_{t_0 \leq s \leq t}\mathcal{E}_s$ for $t_0 \leq t < T$, we have 
\begin{equation}\label{eq:recursion_sol} 
Z_t \leq \left( \prod_{s=0}^{t-t_0-1} \rho_{t_0+s} \right)  \left(1+ \frac{\log^2 n}{\rho_{t_0} n^{1/4}} \right)^{t-t_0} Z_{t_0}. 
\end{equation}

In a multi-type branching process, the variable $Z_t/\rho^t$, where $\rho$ is the largest positive eigenvalue of the progeny matrix, 
is a martingale (see for example Theorem 4 in Chapter V.6 of~\cite{bk:AthNey}). Here, we use this fact only approximately, since the
progeny matrix changes as the process evolves. Nevertheless, after time $t_0$ the matrix does not change immensely. Whereby, 
we are able to control the increase of the largest eigenvalue. 
Let us now make this precise. 

By (\ref{eq:eigenvalue}), the largest positive eigenvalue of $A_{t_0}$ is bounded by a constant $\rho_0 <1$, with probability 
$1-o(1)$. Set $\lambda_{min}:=\min_i \{ \lambda_i\}$. 
For any $t \geq t_0$, let $$\mathcal{D}_t :=
\left\{ \sum_{j=1}^{p_\ell} \sum_{k=2}^r \sum_{v \in \cU(t)} d_{\mathcal{C}_{j,r-k} (t)}(v) < 
\max \left\{ \frac{10 C_\gamma^2}{\lambda_{min} d} Z_t ,n^{1/2}\right\} \right\}. $$
\begin{claim} \label{clm:deg_Conc}
For any $t \geq t_0$ we have $\Pro{\mathcal{D}_t} = 1-o(n^{-1})$.  
\end{claim}
\begin{proof}[Proof of Claim~\ref{clm:deg_Conc}]
The random variable $\sum_{j=1}^{p_\ell} \sum_{k=2}^r \sum_{v \in \cU(t)} d_{\mathcal{C}_{j,r-k} (t)}(v)$ is stochastically bounded 
from above by $\sum_{v \in \cU(t)} X_v$, where the $X_v$s  are i.i.d. random variables that are distributed 
as $\mathrm{Bin} (n,(2C_\gamma)^2 / W_{[n]})$.  The expected value of this sum bounded by ${5 C_\gamma^2 \over d} u(t)$ 
for large $n$. Also, $u(t) \leq Z_t/\lambda_{min}$, as 
$Z_t = (\xi, \overline{u}_{t}) = \sum_{i} \lambda_i u_i(t) \geq \lambda_{min} \sum_i u_i(t)$. 
So the expectation is at most ${5 C_\gamma^2 \over \lambda_{min} d} Z_t$. The claim follows from a standard Chernoff bound 
on the binomial distribution (as the sum of identically distributed binomials is itself binomially distributed). 
\end{proof}
Let $B_t:= \max \left\{ \frac{10 C_\gamma^2}{\lambda_{min} d} Z_t ,n^{1/2}\right\}$.

On the event $\mathcal{D}_t$, the total degree of the vertices in $\cU (t)$ into the set $\mathcal{C}_{j,r-2} (t)$ bounds the 
number of vertices that enter into the set $\mathcal{C}_{j,r-1} (t)$. Hence, on the event $\mathcal{D}_t$, we have 
$$c_{i, r-1}( t +1) \leq c_{i,r-1}(t) + B_t.$$ 
Furthermore, for large $n$
$$u(t+1) \frac{4 C_\gamma^2}{W_{[n]}}\sum_{k=2}^{r} c_{j,r-k} (t+1) \leq u(t+1) \frac{4 C_\gamma^2}{W_{[n]}} n  \leq 
u(t+1) \frac{5 C_\gamma^2}{dn}n = u(t+1) \frac{5 C_\gamma^2}{d}\leq  Z_{t+1} \frac{5 C_\gamma^2}{\lambda_{min} d}.$$
Also, on $\mathcal{E}_t$ we have $Z_{t+1} \leq \beta_1 Z_t$, for some constant $\beta_1 >0$. 
Therefore, on $\mathcal{D}_t \cap \mathcal{E}_t$ we have 
$$\sum_{i=1}^{p_\ell} \frac{W_i^2}{W_{[n]}} a_i ( t +1)
 \leq \sum_{i=1}^{p_\ell} \frac{W_i^2}{W_{[n]}} \left( c_{i, r-1}( t ) +  B_t + Z_{t+1} \frac{5 C_\gamma^2}{\lambda_{min} d}\right) \leq 
\sum_{i=1}^{p_\ell} \frac{W_i^2}{W_{[n]}} c_{i, r-1}( t ) + \beta \frac{B_t}{n}, $$
for some constant $\beta$ and any $n$. 
In other words, 
\begin{equation} \label{eq:eigenineq}
\rho_{t+1} \leq \rho_t + \beta \frac{B_t}{n}. 
\end{equation}
We now prove the following claim. 
\begin{claim} \label{clm:rho_solution} There exists a constant $\beta' >0$ such that the following holds.
Assuming that $\cap_{t_0\leq s \leq t} \{ \mathcal{D}_s \cap \mathcal{E}_s \}$ is realised for a certain $t_0 \leq t < T$, we have 
$$ \rho_t \leq \rho_{t_0} + \beta' \frac{Z_{t_0}}{n} \sum_{s=0}^{t-t_0-2} \left( \prod_{i=0}^s \rho_{t_0+i} \right) 
\left( 1 + \frac{\log^2 n}{\rho_{t_0}n^{1/4}} \right)^{s+1}.$$
\end{claim}
\begin{proof}[Proof of Claim~\ref{clm:rho_solution}]
We show this by induction on $t$. (We shall assume that the empty sum is equal to 0 and the empty product is equal to 1.)
 For $t=t_0$, the statement is obviously true. Suppose that it holds for any $t_0 \leq s<t$. 
 Let $\beta''$ be such that $B_{t-1}/Z_{t-1} < \beta''$, for all $t_0 < t \leq T$ (by the definition of $B_t$ and the stopping time $T$ such
 a constant does exist). Set $\beta' :=\beta \cdot \beta''$. 
Using (\ref{eq:eigenineq}) we have 
\begin{equation*}
\begin{split}
\rho_t &\leq \rho_{t-1} + \beta \frac{B_{t-1}}{n} \leq \rho_{t_0}+ \beta' \frac{Z_{t_0}}{n} 
\sum_{s=0}^{t-t_0-3} \left( \prod_{i=0}^s \rho_{t_0+i} \right) 
\left( 1 + \frac{\log^2 n}{\rho_{t_0} n^{1/4}} \right)^{s+1} + \beta \frac{B_{t-1}}{n} \\
& \leq \rho_{t_0}+ \beta' \frac{Z_{t_0}}{n} \sum_{s=0}^{t-t_0-3} \left( \prod_{i=0}^s \rho_{t_0+i} \right) 
\left( 1 + \frac{\log^2 n}{\rho_{t_0} n^{1/4}} \right)^{s+1} + \beta' \frac{Z_{t-1}}{n} \\
&\stackrel{(\ref{eq:recursion_sol})}{\leq}
\rho_{t_0}+ \beta' \frac{Z_{t_0}}{n} \sum_{s=0}^{t-t_0-3} \left( \prod_{i=0}^s \rho_{t_0+i} \right) 
\left( 1 + \frac{\log^2 n}{\rho_{t_0} n^{1/4}} \right)^{s+1} \\ 
&\hspace{2cm}+ \beta' \frac{Z_{t_0}}{n} \left( \prod_{i=0}^{t-1-t_0-1} \rho_{t_0+i} \right) 
\left( 1 + \frac{\log^2 n}{\rho_{t_0} n^{1/4}} \right)^{t-1-t_0}.
\end{split}
\end{equation*}
\end{proof}

We now show inductively that $\rho_t$ is uniformly bounded by some constant that is less than 1, as long as $\mathcal{D}_s$ and 
$\mathcal{E}_s$ are realised for all $t_0 \leq s \leq t$. 
Here, we will require that $Z_{t_0} /n$ is small enough, which we can assume as this quantity is proportional to 
$\epsilon$. 
\begin{claim} \label{clm:eigenboundfinal}
For any $\delta >0$, there exists an $\eps$ such that the following holds for $t \leq T \wedge (t_0+ \log^2 n)$. If $Z_{t_0}/n < \eps$, then
provided that $\cap_{t_0\leq s \leq t} \{ \mathcal{D}_s \cap \mathcal{E}_s \}$ is realised we have 
$$ \rho_t < \rho_{t_0} + \delta. $$
\end{claim} 
\begin{proof}[Proof of Claim~\ref{clm:eigenboundfinal}] 
We will show this by induction. 
Clearly $\rho_0$ satisfies the inequality. Assume now that this holds for all $t_0 \leq s<t$, that is, $\rho_s < \rho_{t_0}+\delta$. 
Then by Claim~\ref{clm:rho_solution} we have 
\begin{equation*}
\begin{split} 
\rho_t &\leq \rho_{t_0} + \beta' \frac{Z_{t_0}}{n} \sum_{s=0}^{t-t_0-2} \left( \prod_{i=0}^s \rho_{t_0+i} \right) 
\left( 1 + \frac{\log^2 n}{ \rho_{t_0} n^{1/4}} \right)^{s+1} \\
& \leq \rho_{t_0} + \beta' \frac{Z_{t_0}}{n} \left( 1 + \frac{\log^2 n}{\rho_{t_0} n^{1/4}} \right)^{t-t_0} \sum_{s=0}^{t-t_0-2} 
\left( \prod_{i=0}^s \rho_{t_0 + i} \right) \\
& \stackrel{\rho_{t_0+i} < \rho_{t_0} + \delta}{\leq} 
\rho_{t_0} + \beta' \frac{Z_{t_0}}{n} \left( 1 + \frac{\log^2 n}{\rho_{t_0} n^{1/4}} \right)^{t-t_0} 
\sum_{s=0}^{t-t_0}  \left( \rho_{t_0}+\delta \right)^{s+1} \\
&\leq \rho_{t_0} + \beta' \frac{Z_{t_0}}{n} \left( 1 + \frac{\log^2 n}{\rho_{t_0} n^{1/4}} \right)^{t-t_0} 
\frac{1}{1- \rho_{t_0} -\delta} \\
&\stackrel{t -t_0\leq \log^2 n}{\leq} \rho_{t_0} + \beta' \frac{Z_{t_0}}{n} \frac{1 + o(1)}{1- \rho_{t_0} -\delta} 
 \leq \rho_{t_0} + \beta' \eps~\frac{1+o(1)}{1-\rho_{t_0} -\delta } < \rho_{t_0} +\delta,
\end{split}
\end{equation*}
if $\eps$ is small enough. 
\end{proof}
The above claim together with Claims~\ref{clm:conc_proc} and~\ref{clm:deg_Conc}, it follows that 
$T < \log_{\frac{1}{\rho_{t_0}+\delta}} n$ with probability $1-o(n^{-1})$. 
Note that conditional on the event $\cap_{t_0 \leq s \leq T} \mathcal{D}_s \cap \mathcal{E}_s$, the total 
number of a vertices infected until time $T$ is proportional to $u(t_0)$. 

Thereafter, the process is stochastically bounded from above by a sub-critical multi-type branching process. 
Let $T_1 = \min \{ t > T \ : \ Z_t> n^{1/2}\}$. Arguing as above, we can show that for any $T \leq t < T_1$, 
with conditional probability $1-o(n^{-1})$ we have 
$$ Z_{t+1} \leq  n^{1/2}.$$ 
We denote this event by $\mathcal{E}_t'$. 
Consider the time window $I_f:=[T, T_1\wedge (T+\log^2 n)]$. Then the event $\cap_{t \in I_f} \mathcal{E}_t'$ occurs with 
probability $1-o(1)$. 
Conditional on this event, for any $t$ in this window the process is stochastically bounded from above by a multi-type branching process
where the largest positive eigenvalue is bounded from above by $\rho_{t_0} + 2 \delta =:\rho_f < 1$. 
This follows from (\ref{eq:eigenineq}) and Claim~\ref{clm:eigenboundfinal}, for any $n$ sufficiently large, 
since $\mathcal{D}_t$ occurs with probability $1-o(n^{-1})$. 
By Theorem 4 in Section V.4 of~\cite{bk:AthNey} (in fact, by the proof of it) there exists a constant $Q>0$ such that 
$$ \PRO { \overline{u} (T+i) \not = 0 \ | \overline{u}(T)}= Q~ (\overline{u}(T),\xi)~\rho_f^i  + O(\rho_f^{2i}). $$
Therefore for $i= \log_{1/\rho_f} n$, the above probability is $o(1)$ uniformly over every realisation of $\overline{u}(T)$.

\subsection{Auxiliary lemmas}

Recall that $\hat{\tau}^{(\ell, \gamma)}$ denotes the minimum $\tau >0$ such that $\mu_{\cU} (\tau) = 0$.
Recall also that $\hat{y}$ is the smallest positive solution of $f_r(y;W_F^*,p)=0$ and that we have assumed that
$f_r'(\hat{y};W_F^*,p) < 0$. Also, recall that
$$ \alpha (y ) := p~(1-\gamma ) +\gamma'  +   (1-p)
\sum_{i=1}^{p_\ell}  \gamma_i  \psi_r \left( W_i y \right).$$
The following lemma shows that if $\gamma$ is taken small enough and $\ell$ is a large positive integer,
then $\alpha (\hat{y}_{\ell,\gamma})$ and $\mu_{\cU}'(\hat{\tau}^{(\ell,\gamma)})$ can be approximated
by the corresponding functions of $\hat{y}$.
 For technical reasons, we need to restrict ourselves to those $\gamma$s for which $1-\gamma \in F([0,\infty) )$ --
we will be referring to such a $\gamma$ as \emph{being in the range of $F$}.
\begin{lemma} \label{lem:stableSol}
Assume that $f_r'(\hat{y};W_F^*,p) < 0$.
For $\gamma >0$, let $\{ (\bW^{(\ell,\gamma)}(n))_{n\geq 1}\}_{\ell \in \mathbb{N}}$ be an $F$-convergent 
$(\ell, \gamma)$-discretisation of the weight sequence $(\mathbf{w}(n))_{n \geq 1}$ with error $\rho>0$.
Then there exists $\gamma_2$ having the property that for any $\gamma < \gamma_2$ which
is in the range of $F$ and any $\delta > 0$, there exists a subsequence $\{ \ell_k \}_{k \in \mathbb{N}}$ such that
for every $\ell \in \{ \ell_k \}_{k \in \mathbb{N}}$:
\begin{enumerate}
\item $\mu_{\cU}' (\hat{\tau}^{(\ell,\gamma)}) < 0$;
\item $\left|\alpha (\hat{y}_{\ell,\gamma }) - \left(p + (1-p)\mathbb{E}(\psi_r (W_F \hat{y})) \right) \right| < \delta$.
\end{enumerate}
\end{lemma}
\begin{proof}
As above, we have set $x=x (\tau) = I(\tau) /d$.
Since $I(\tau) = \int_0^{\tau} F(s) ds$, we have $x' (\tau)  >0$, for all $\tau >0$. Thereby, setting $\hat{\mu}_{\cU} (\tau)
= \mu_{\cU}(\tau )/d$, it suffices to show that $\hat{\mu}_{\cU}' (x(\hat{\tau}^{(\ell,\gamma)})) <0$.
Recall that $\hat{y}_{\ell,\gamma} = x (\hat{\tau}^{(\ell,\gamma)})$.

We will use (\ref{eq:probs}) in order to express the $\gamma_i$s in terms of the $\gamma_i'$s:
$\gamma_i = (1-\gamma +\gamma')\gamma_i'$.
The expression for $\mu_{\cU}$ as it is given in Proposition~\ref{eq:uSolution} yields the following
\begin{equation} \label{eq:FixedPoint1st}
\begin{split}
&\hat{\mu}_{\cU} (x )  = {W_{\gamma}'\over d} +p~ {1\over d}\sum_{i=1}^{p_\ell} W_i \gamma_i  -x +
(1-p)~{1\over d} \sum_{i=1}^{p_\ell}  W_i \gamma_i \PRO {\Po \left( W_i  x  \right) \geq r} \\
&= {W_{\gamma}'\over d} +(1-\gamma + \gamma' ) \left( p~ {\hat{d}^{(\ell,\gamma )}\over d}   +
(1-p)~{d^{(\ell,\gamma)}\over d} \sum_{i=1}^\ell  {W_i \gamma_i' \over d^{(\ell,\gamma)}} \PRO {\Po \left( W_i  x  \right) \geq r}
\right) -x,
\end{split}
\end{equation}
where $d^{(\ell,\gamma)} = \int_{0}^{\infty} x dF^{(\ell,\gamma)}(x)$ and
$\hat{d}^{(\ell,\gamma)} = \int_{0}^{C_\gamma} x dF^{(\ell,\gamma)}(x)= \sum_{i=1}^{p_\ell} W_i \gamma_i$.
Hence, the second sum in the above expression can be rewritten as
\begin{equation*}
 \sum_{i=1}^{p_\ell}  {W_i \gamma_i' \over d^{(\ell,\gamma)}} \PRO {\Po \left( W_i  x  \right) \geq r} =
\int_{0}^{C_\gamma} \psi_r \left( y  x  \right) dF^{*(\ell,\gamma)}(y),
\end{equation*}
where $F^{*(\ell, \gamma)}$ is the distribution function of the $U^{(\ell,\gamma)}$ size-biased distribution.

We set $c(\gamma) = 1 -\gamma +\gamma'$ and
write $p^{(\ell,\gamma)} =  {W_{\gamma}'\over d c(\gamma)} +p~ {\hat{d}^{(\ell,\gamma)}\over d}$.
The expression in (\ref{eq:FixedPoint1st})  becomes
\begin{equation*} 
\hat{\mu}_{\cU} (x) = c(\gamma) \left( p^{(\ell,\gamma)}  + (1-p)~{d^{(\ell,\gamma)}\over d}~
\int_{0}^{C_\gamma} \psi_r \left( y  x  \right) dF^{*(\ell,\gamma)}(y) \right) - x.
\end{equation*}
Hence, the derivative of $\hat{\mu}_{\cU} (x)$ is
\begin{equation} \label{eq:muder}
\begin{split}
\hat{\mu}_{\cU}'(x) &= -1 + c(\gamma ) (1-p)~{d^{(\ell,\gamma)} \over d}~
\int_{0}^{C_\gamma} y e^{-y x}
{\left( y x \right)^{r-1} \over (r-1)!} dF^{*(\ell,\gamma)} (y)  \\
&= -1 + c(\gamma) (1-p)~{d^{(\ell,\gamma )} \over d}~{r \over x}~
\int_{0}^{C_\gamma} e^{-y x}
{\left( y x \right)^r \over r!} dF^{*(\ell,\gamma)}(y).
\end{split}
\end{equation}
Similarly, we can write
\begin{equation} \label{eq:alpha_new}
\alpha ( x ) = p~(1-\gamma ) +\gamma'  +   c(\gamma ) (1-p)
\int_{0}^{C_\gamma}\psi_r \left( y x \right) dF^{(\ell,\gamma)} (y).
\end{equation}

For real numbers $y$ and $\delta >0$, let $B(y;\delta)$ denote the open ball of radius $\delta$ around $y$.
We will later show the following statement.
\begin{proposition} \label{prop:cts_approx}
Let $f:\mathbb{R} \rightarrow \mathbb{R}$ be a bounded function which is everywhere differentiable, has bounded derivative and 
satisfies $f(0)=0$.
Let also $y_1 \in \mathbb{R}$.
For any $\delta>0$ there exists $\gamma_3 = \gamma_3(\delta)$ with the property that for any $\gamma < \gamma_3$ in the range of
$F$, there exist $\ell_0 = \ell_0(\delta, \gamma) >0$ and $\delta' = \delta' (\delta, \gamma)$ such that for any $\ell > \ell_0$ and any
$y_2 \in B(y_1;\delta')$,
\begin{equation*}
\left| \int_{0}^{C_\gamma} f(y y_2) dF^{*(\ell,\gamma)}(y) - \mathbb{E} \left( f(W_F^* y_1) \right)\right|
< \delta,
\end{equation*}
and
\begin{equation*}
\left| \int_{0}^{C_\gamma} f(y y_2) dF^{(\ell,\gamma)}(y) - \mathbb{E} \left( f(W_F y_1) \right)\right|
< \delta.
\end{equation*}
\end{proposition}
We will further show that $\hat{y}_{\ell,\gamma}$ is close to $\hat{y}$ over a subsequence $\{\ell_k \}_{k\in \mathbb{N}}$.
\begin{proposition} \label{prop:roots_approx}
There exists a $\gamma_4>0$ such that for all $\gamma < \gamma_4$ and any $\delta' > 0$ there exists a subsequence
$\{\ell_k \}_{k \in \mathbb{N}}$ such that
$\hat{y}_{\ell_k,\gamma} \in B(\hat{y};\delta').$
\end{proposition}
The above two propositions yield the following:
\begin{corollary} \label{cor:cts_approx}
Let $f:\mathbb{R} \rightarrow \mathbb{R}$ be a bounded function which is everywhere differentiable and has bounded derivative.
For any $\delta>0$, any
$\gamma < \gamma_3 \wedge \gamma_4$, which is in the range of $F$, there exists a
subsequence $\{ \ell_k \}_{k \in \mathbb{N}}$ such that
\begin{equation*}
\left|\int_{0}^{C_\gamma} f(y \hat{y}_{\ell_k, \gamma}) dF^{*(\ell_k,\gamma)}(y) - \mathbb{E} \left( f(W_F^* \hat{y}) \right)\right|
< \delta,
\end{equation*}
and
\begin{equation*}
\left|\int_{0}^{C_\gamma} f(y \hat{y}_{\ell_k, \gamma}) dF^{(\ell,\gamma)}(y) - \mathbb{E} \left( f(W_F \hat{y}) \right)\right|
< \delta.
\end{equation*}
\end{corollary}
The two statements of the lemma can be deduced from (\ref{eq:muder}) and (\ref{eq:alpha_new}), if we let $f(x)$ be $\psi_r (x)$ in
the former case, and $e^{-x}{x^{r} \over r!}$ in the latter. Note that the choice of the subsequence is determined through
Proposition~\ref{prop:roots_approx} and can be the same for both choices of $f(x)$.

Observe that both functions are bounded (by 1), they are differentiable and have bounded
derivatives.
By the second part of Definition~\ref{def:F-conv} and the fact that $c(\gamma ) \rightarrow 1$ as $\gamma \downarrow 0$ we have
\begin{equation} \label{eq:II} c(\gamma) \left|{d^{(\ell, \gamma)} \over d} -1 \right| < \delta, \end{equation}
for any $\gamma$ that is small enough and any $\ell$ that is large enough.
Both parts of the lemma now follow from Corollary~\ref{cor:cts_approx} together with (\ref{eq:II}). 

We now proceed with the proofs of Propositions~\ref{prop:cts_approx} and~\ref{prop:roots_approx}. 
There, we shall need the following claim, which shows that $p^{(\ell,\gamma)}$ is close to $p$. 
\begin{claim} \label{clm:hat_d_approx}
There is a function $r: \mathbb{R}^+ \rightarrow \mathbb{R}^+$ such that $r(\gamma ) \rightarrow 0$ as $\gamma \downarrow 0$
with the following property. Let $\gamma_0$ and $L_1 (\gamma)$ be as in Definition~\ref{def:F-conv}. For any $0 < \gamma < \gamma_0$ and $\ell >  L_1 (\gamma)$ 
\begin{equation*} 
\big| \hat{d}^{(\ell,\gamma)} - d \big| < r(\gamma ).
\end{equation*}
\end{claim}
\begin{proof}
From the definition we obtain that
$$d = \int_{0}^{\infty} x dF(x) = \int_{0}^{C_\gamma} x dF(x) + W_\gamma. $$
Note that $W_\gamma$ tends to 0 as $\gamma \downarrow 0$. For the integral on the right-hand side we use again the integration-by-parts formula and obtain
\begin{equation} \label{eq:1st_Int_I}
\begin{split}
\int_{0}^{C_\gamma} x dF(x)
&= F(C_\gamma+) C_\gamma   - \int_{0}^{C_\gamma} F(x) dx .
\end{split}
\end{equation}
Similarly, we write
\begin{equation} \label{eq:2nd_Int_II}
\begin{split}
\hat{d}^{(\ell,\gamma)} = \int_{0}^{C_\gamma}  x dF^{(\ell,\gamma)}(x) =
 F^{(\ell,\gamma)} (C_\gamma +) C_\gamma -  \int_{0}^{C_\gamma} F^{(\ell,\gamma)}(x) dx.
\end{split}
\end{equation}
The first part of Definition~\ref{def:F-conv} implies that if $0 < \gamma < \gamma_0$ and $\ell > L_1 (\gamma)$, then since
both $F^{(\ell,\gamma)}, F$ are right-continuous we have
$$|F^{(\ell, \gamma)} (C_\gamma+) - F(C_\gamma+)| = |F^{(\ell, \gamma)} (C_\gamma) - F(C_\gamma)| < 2(\gamma + W_\gamma / C_\gamma) =: y(\gamma). $$
Thus, (\ref{eq:1st_Int_I}) and (\ref{eq:2nd_Int_II}) together yield
\begin{equation*}
\begin{split}
&\left|\int_{0}^{C_\gamma}  x dF(x)  -  \int_{0}^{C_\gamma} x d F^{(\ell,\gamma)}(x) \right| \leq  y(\gamma)  C_\gamma +
\int_{0}^{C_\gamma} |F(x) - F^{(\ell,\gamma)} (x) |  d x \\
&  \qquad  \stackrel{(\mbox{{\small Def.}}~\ref{def:F-conv})}{<}
y(\gamma) C_\gamma + y (\gamma) \left( C_{\gamma}  - 0  \right).
\end{split}
\end{equation*}
We finally choose $r(\gamma ) = W_\gamma + 2y(\gamma)C_\gamma$; due to~\eqref{eq:convpr} this tends to 0 as  $\gamma \downarrow 0$.
\end{proof}

\begin{proof}[Proof of Proposition~\ref{prop:cts_approx}]
The proof of this proposition will proceed in two steps.
Firstly, we will show that for any $\gamma < \gamma_1$ (cf. Lemma~\ref{lem:size_biased_approx}) there exist
$\delta' = \delta'(\delta, \gamma)$ and $\ell_0 = \ell_0 (\delta , \gamma)$ such that for any $y_2 \in B(y_1;\delta')$ and
$\ell > \ell_0$ we have
\begin{equation} \label{eq:target_I}
\left| \int_{0}^{C_\gamma} f(y y_2) dF^{*(\ell,\gamma)}(y) -
\int_{0}^{C_\gamma} f(y y_1) dF^{*}(y) \right| < \delta /2.
\end{equation}
The proposition will follow if show that there exists $\gamma_2 = \gamma_2 (\delta)$ such that for any $\gamma < \gamma_2$
it holds that
\begin{equation}  \label{eq:target_II}
\left| \int_{C_\gamma}^{\infty} f(y y_1) dF^{*}(y)  \right| < \delta /2.
\end{equation}
If we show these inequalities, then we deduce
\begin{equation*}
\begin{split}
& \left| \int_{0}^{C_\gamma} f(y y_2) dF^{*(\ell,\gamma)}(y)- \ex {f(W_F^*y_1)} \right| \\
\le &\left| \int_{0}^{C_\gamma} f(y y_2) dF^{*(\ell,\gamma)}(y) -
\int_{0}^{C_\gamma} f(y y_1) dF^{*}(y) \right| + \left| \int_{C_\gamma}^{\infty} f(y y_1) dF^{*}(y) \right| 
 \stackrel{(\ref{eq:target_I}), (\ref{eq:target_II})}{<} \delta.
\end{split}
\end{equation*}
The proof for the case of $U^{(\ell,\gamma)}$ proceeds along the same lines.
We can show that for any $\gamma < \gamma_1$ there exist
$\delta' = \delta'(\delta, \gamma)$ and $\ell_0 = \ell_0 (\delta , \gamma)$ such that for any $y_2 \in B(y_1;\delta')$ and
$\ell > \ell_0$ we have
\begin{equation} \label{eq:target_I_II}
\left| \int_{0}^{C_\gamma} f(y y_2) dF^{(\ell,\gamma)} (y) -
\int_{0}^{C_\gamma} f(y y_1) dF(y) \right| < \delta /2.
\end{equation}
Then we show that there exists $\gamma_2 = \gamma_2 (\delta)$ such that for any $\gamma < \gamma_2$
it holds that
\begin{equation}  \label{eq:target_II_II}
\left| \int_{C_\gamma}^{\infty} f(y y_1) dF(y) \right| < \delta /2.
\end{equation}
As before, from (\ref{eq:target_I_II}) and (\ref{eq:target_II_II}) we deduce
\begin{equation*}
\begin{split}
& \left| \int_{0}^{C_\gamma} f(y y_2) dF^{(\ell,\gamma)} (y) - \ex {f(W_F y_1)} \right| \\
\le & \left| \int_{0}^{C_\gamma} f(y y_2) dF^{(\ell,\gamma)} (y) -
\int_{0}^{C_\gamma} f(y y_1) dF(y) \right|  + \left| \int_{C_\gamma}^{\infty} f(y y_1) dF(y) \right|  \stackrel{(\ref{eq:target_I_II}), (\ref{eq:target_II_II})}{<} \delta.
\end{split}
\end{equation*}

We proceed with the proofs of (\ref{eq:target_I}) and (\ref{eq:target_II}) -- the proofs of (\ref{eq:target_I_II}) and
(\ref{eq:target_II_II}) are very similar (in fact, simpler) and are omitted.
\begin{proof}[Proof of (\ref{eq:target_I})]
We begin with the specification of $\delta'$. We let $\delta'$ be such that whenever $|y_1 - y_2 | < \delta'$ we have
\begin{equation} \label{eq:mod_of_continuity}
\left|  f(x y_1) -  f(x y_2) \right| < \delta/4,
\end{equation}
for any $x \in [x_0,C_\gamma]$. This choice of $\delta'$ is possible since $f$ is continuous and therefore uniformly continuous in any
closed interval. Consider $y_2 \in B(y_1;\delta')$. We then have
\begin{equation} \label{eq:Intermediate_I}
\begin{split}
& \left| \int_{0}^{C_\gamma} f(x y_2) d F^{*(\ell,\gamma)} (x) -
\int_{0}^{C_\gamma} f(x y_1) d F^*(x)  \right|\\
\leq & \int_{0}^{C_\gamma} |f(xy_2) - f(xy_1)|  dF^{*(\ell,\gamma)} (x) \\
& \hspace{2.5cm}+ \left|\int_{0}^{C_\gamma} f(x y_1) d F^{*(\ell,\gamma)} (x)  -
\int_{0}^{C_\gamma} f(x y_1) d F^*(x)\right| \\
\stackrel{\eqref{eq:mod_of_continuity}}{\leq} & \delta/4 +
\left|\int_{0}^{C_\gamma} f(x y_1) d F^{*(\ell,\gamma)} (x)  -
\int_{0}^{C_\gamma} f(x y_1) d F^*(x)\right|.
\end{split}
\end{equation}
We will argue that the second expression is also bounded from above by $\delta/4$ when $\gamma$ is small enough.
Since, $f$ as well as $F^*$ and $F^{*(\ell,\gamma)}$ have bounded variation and $f$ is continuous, we can use the
integration-by-parts formula for each one of the two integrals. We have
\begin{equation*}
\begin{split}
\int_{0}^{C_\gamma} f(x y_1) d F^{*(\ell,\gamma)} (x)
& = f(C_\gamma y_1) F^{*(\ell,\gamma)} (C_\gamma + ) - f(0 y_1) F^{*(\ell,\gamma)} (0-) - \int_{0}^{C_\gamma}
F^{*(\ell, \gamma)} (x) d f(xy_1)  \\
&\stackrel{(f(0)=0)}{=} f(C_\gamma y_1) F^{*(\ell,\gamma)} (C_\gamma + ) - \int_{0}^{C_\gamma}
F^{*(\ell, \gamma)} (x) d f(xy_1)
\end{split}
\end{equation*}
and
\begin{equation*}
\begin{split}
\int_{0}^{C_\gamma} f(x y_1) d F^{*} (x)
& = f(C_\gamma y_1) F^{*} (C_\gamma + ) - f(0) F^{*} (0-) - \int_{0}^{C_\gamma}
F^{*} (x) d f(xy_1)  \\
& \stackrel{(f(0)=0)}{=} f(C_\gamma y_1) F^{*} (C_\gamma + ) - \int_{0}^{C_\gamma}
F^{*} (x) d f(xy_1).
\end{split}
\end{equation*}
Thereby, we have
\begin{equation*}
\begin{split}
&\left|\int_{0}^{C_\gamma} f(x y_1) d F^{*(\ell,\gamma)} (x)  -
\int_{0}^{C_\gamma} f(x y_1) d F^*(x)\right| \leq \\
&| f(C_\gamma y_1)| ~ \left|F^{*(\ell,\gamma)} (C_\gamma + ) - F^* (C_\gamma + ) \right| +
\int_{0}^{C_\gamma} |F^{*(\ell,\gamma)} (x) - F^*(x) | d f(xy_1).
\end{split}
\end{equation*}
By Lemma~\ref{lem:size_biased_approx} there exists $\gamma_1$ such that for any $\gamma < \gamma_1$,
for almost all $x \in [x_0, C_\gamma ]$ we have
$$ |F^{*(\ell,\gamma)} (x) - F^*(x) | < \rho_1 (\gamma), $$
for any $\ell$ that is large enough (depending on $\gamma$).
Additionally, for the set of $x$s of measure 0 where this does not hold, the difference is bounded by 1.
As $f$ is differentiable and, therefore, continuous everywhere, the second integral is bounded by $\rho_1 (\gamma) |f(C_\gamma y_1)|$.
Therefore,
\begin{equation*}
\left|\int_{0}^{C_\gamma} f(x y_1) d F^{*(\ell,\gamma)} (x)  -
\int_{0}^{C_\gamma} f(x y_1) d F^*(x)\right| \leq 2| f(C_\gamma y_1)|~ \rho_1 (\gamma).
\end{equation*}
Since $f$ is bounded, if $\gamma$ is small enough, then the latter expression is at most $\delta /4$.
Hence (\ref{eq:target_I}) follows if we substitute this bound into (\ref{eq:Intermediate_I}).
\end{proof}
We now proceed with the proof of (\ref{eq:target_II}).
\begin{proof}[Proof of (\ref{eq:target_II})]
Assume that  $|f(x)| < b$ for any $x \in \mathbb{R}$. Hence we have
\begin{equation}\label{eq:I}
\left| \int_{C_\gamma}^{\infty} f(y y_1) dF^{*}(y)  \right| < b \ex {\mathbf{1}_{\{W_F^{*} \geq C_\gamma\}}}.
\end{equation}
We bound the latter indicator function from above by the function
\begin{equation*} h(x) =
\begin{cases}
2 - e^{-(x- C_\gamma)}, & \mbox{$x>C_\gamma - \ln 2$} \\
0, & \mbox{otherwise}
\end{cases}.
\end{equation*}
It is easy to see that for any $x > C_\gamma $ the above function exceeds 1, whereas for any other $x$ it is non-negative.
It hits 0 at $x= C_\gamma -\ln 2$. Also, observe that if $x\rightarrow \infty$, then $h(x)$ approaches 2 from below.
Hence, $h$ is bounded and continuous.

Now, we have
$$  \ex {\mathbf{1}_{\{W_F^{*} \geq C_\gamma\}}} \leq  \ex {h(W_F^*)} = {\ex {W_F h(W_F)}\over d},$$
by the definition of the $W_F$ size-biased distribution (\ref{eq:Size_biased}).
We will bound $\ex {W_F h(W_F)}$ as follows:
$$\ex {W_F h(W_F)} \leq 2\ex{\mathbf{1}_{\{ W_F > C_\gamma - \ln 2 \}} W_F} < d \delta /(4b), $$
provided that $\gamma$ is small enough.
Therefore,
$$ \ex {\mathbf{1}_{\{W_F^{*} \geq C_\gamma\}}} \leq  {\delta \over 4 b}, $$
and (\ref{eq:target_II}) follows from (\ref{eq:I}).
\end{proof}
\end{proof}

\begin{proof}[Proof of Proposition~\ref{prop:roots_approx}]
Recall that $\hat{y}_{\ell,\gamma}$ is the smallest positive root of
$$\hat{\mu}_{\cU}(x)=0.$$
To emphasize the dependency of $\hat{\mu}_{\cU}$ on $\ell, \gamma$, we will set $f_r^{(\ell,\gamma)} (x) := \hat{\mu}_{\cU} (x)$;
thus
$$
f_r^{(\ell,\gamma)} (x )
= {W_{\gamma}'\over d} +p~ c(\gamma ){\hat{d}^{(\ell,\gamma )}\over d}  - x +
(1-p)~c(\gamma ){d^{(\ell,\gamma)}\over d} \int_{0}^{C_\gamma} \psi_r \left( y  x  \right) dF^{*(\ell,\gamma)}(y).
$$

We consider the functions $f_r^{(\ell,\gamma)} (x)$ restricted on the unit interval $[0,1]$.
\begin{claim}\label{clm:eqcts_I}
There exists $\gamma_4>0$ such that for any $\gamma < \gamma_4$ the family
$$\left\{ f_r^{(\ell,\gamma)} (x) \right\}_{\ell > \ell_1},$$
for some $\ell_1 = \ell_1 (\gamma)$, is equicontinuous.
\end{claim}
\begin{proof}[Proof of Claim~\ref{clm:eqcts_I}]
Let $\eps \in (0,1)$ and let $\gamma_4 < \gamma_0$ (cf. Definition~\ref{def:F-conv}) be such that for any $\gamma < \gamma_4$ we have $1/C_\gamma < \eps /2$.
Recall that $\{ (\bW^{(\ell, \gamma)}(n))_{n\geq 1} \}_{\ell \in \mathbb{N}}$ is $F$-convergent with error $\rho >0$ (cf.~Definition~\ref{def:F-conv}).
Fixing such a $\gamma$, for any $\ell > L_1(\gamma )$ (cf. Definition~\ref{def:F-conv})
\begin{equation} \label{eq:av_deg_approx} \left| {d^{(\ell, \gamma)} \over d} - 1 \right| < \rho /d. \end{equation}
The function $\psi_r (y)$ is uniformly continuous on the closed interval $[x_0,C_\gamma]$. Hence there exists $\delta \in (0,1)$ such that
for any $x_1, x_2 \in [0,1]$ with $|x_1-x_2| < \delta / C_\gamma$ we have $|\psi_r (w x_1) - \psi_r (w x_2)| < d \eps/(2\rho)$.
Thus,
\begin{equation} \label{eq:3rdApp_I}
\left| \int_{0}^{C_\gamma} \psi_r \left( y x_1\right) dF^{*(\ell,\gamma)} (y) -
\int_{0}^{C_\gamma} \psi_r \left( y x_2\right) dF^{*(\ell,\gamma)} (y) \right| < \frac{\eps}{2(1+\rho /d)}.
\end{equation}
Thereby, for any $\gamma < \gamma_4$ and $\ell > L_1(\gamma )$ (which we take as $\ell_1 (\gamma)$),
if $x_1, x_2 \in [0,1]$ are such that $|x_1-x_2| < \delta /
C_\gamma$,
then
\begin{equation*}
\begin{split}
|f_r^{(\ell,\gamma)}(x_1) - f_r^{(\ell, \gamma)}(x_2)| 
& \leq  \left|x_1 - x_2 \right| +
 {d^{(\ell,\gamma)}\over d} \left| \int_{0}^{C_\gamma} \psi_r \left( y x_1\right) dF^{*(\ell,\gamma)(y)} -
\int_{0}^{C_\gamma} \psi_r \left( y x_2\right) dF^{*(\ell,\gamma)(y)}  \right| \\
& \stackrel{1/C_\gamma < \eps /2, (\ref{eq:av_deg_approx}), (\ref{eq:3rdApp_I})}{\leq}
{\delta \over C_\gamma} + \frac{\eps}{2(1+\rho /d)}~(1+\rho/d)  \leq  {\eps \over 2} + {\eps \over 2} = \eps.
\end{split}
\end{equation*}
\end{proof}
By the Arzel\'a-Ascoli Theorem, there exists a subsequence $\{ \ell_k\}_{k \in \mathbb{N}}$ such that
$$ \left\{ f_r^{(\ell_k,\gamma)} (x) \right\}_{k \in \mathbb{N}}$$
is convergent in the $L_\infty$-norm on the space of all continuous real-valued functions on $[0,1]$.

Now, recall that $\hat{y}$ is the smallest positive root of $f_r (y; W_F^*,p)=0$ and,  moreover, $f_r'(\hat{y};W_F^*,p) <0$.
Hence,  there exists $\delta_0 >0$ such that
\begin{equation*}
\begin{split}
f_r  (\hat{y} &+\delta_0;W_F^*,p) < 0 \ \mbox{and} \\
& f_r (\hat{y} -\delta_0;W_F^*,p) > 0.
\end{split}
\end{equation*}
Applying Proposition~\ref{prop:cts_approx}, we deduce that there exists $\ell_2 = \ell_2 (\delta_0,\gamma)$ with the property that for any
$k$ such that $\ell_k > \ell_2 $ we have
\begin{equation*}
\begin{split}
f_r^{(\ell_k,\gamma)}  (\hat{y} &+\delta_0) < 0 \ \mbox{and} \\
& f_r^{(\ell_k,\gamma)} (\hat{y} -\delta_0) > 0.
\end{split}
\end{equation*}
In turn, this implies that for any such $k$ there exists a root of $f_r^{(\ell_k,\gamma)} (x)$ in $B(\hat{y};\delta_0)$.

To conclude the proof of the proposition, we need to show that there is no root of $f_r^{(\ell_k,\gamma)}$ in the interval
$[0,\hat{y}-\delta_0]$. Assume, for the sake of contradiction, that there exists a sub-subsequence $\{\ell_{k_i} \}_{i \in \mathbb{N}}$
such that $\hat{y}_{\ell_{k_i},\gamma} \in [0,\hat{y}- \delta_0]$. By the sequential compactness of this interval, we deduce
that there is a further sub-subsequence $\{ \ell_{k_j} \}_{j \in \mathbb{N}}$ over which
$$ \hat{y}_{\ell_{k_j},\gamma} \rightarrow \hat{y}_\gamma, $$
as $j \rightarrow \infty$, for some $\hat{y}_\gamma \in [0,\hat{y}-\delta_0]$.

We will show that $\hat{y}_\gamma =0$. Assume that this is not the case.
Let $\delta \in (0,1)$ and let $\gamma_3 = \gamma_3(\delta)$ be as in Proposition~\ref{prop:cts_approx}.
Consider a $\gamma < \gamma_3$ in the range of $F$.
 Then there exists $j_0$ such that
for $j > j_0$ we have
\begin{equation} \label{eq:subs_approx_I}
\left|\int_{0}^{C_\gamma} f(y \hat{y}_{\ell_{k_j},\gamma}) dF^{*(\ell_{k_j},\gamma)}(y)
- \mathbb{E} \left( f(W_F^* \hat{y}_\gamma) \right)\right|
< \delta/3.
\end{equation}
Assume that $\gamma$ is small enough so that
$$ |c(\gamma ) - 1|, \ \rho (\gamma ) /d, \ r(\gamma ) /d < \delta /9 \min \{ \left\| f \right\|_{\infty}^{-1} ,1 \}.  $$
Moreover, assume that $j_0$ is large enough so that for $j > j_0$  we have
\begin{equation*}
\left| {d^{(\ell_{k_j}, \gamma)} \over d} - 1 \right| < \rho (\gamma ) /d \ \mbox{and} \
 \left| {\hat{d}^{(\ell_{k_j}, \gamma)} \over d} - 1 \right|  < r(\gamma ) /d,
\end{equation*}
by Definition~\ref{def:F-conv} and Claim~\ref{clm:hat_d_approx}.
Hence
\begin{equation}\label{eq:subs_approx_II}
\begin{split}
\left|c(\gamma ){\hat{d}^{(\ell_{k_j}, \gamma)} \over d} - 1  \right| &\leq
\left| c(\gamma ) - 1\right|{d^{(\ell_{k_j}, \gamma)} \over d}
+ \left| {\hat{d}^{(\ell_{k_j}, \gamma)} \over d} - 1 \right|  \\
&\leq \left| c(\gamma ) - 1\right| + \left| c(\gamma ) - 1\right|  \left| {d^{(\ell_{k_j}, \gamma)} \over d} - 1 \right|  +
\left| {\hat{d}^{(\ell_{k_j}, \gamma)} \over d} - 1 \right| \\
&\leq 3 \delta /9 = \delta /3.
\end{split}
\end{equation}
Similarly, we can show that
\begin{equation}\label{eq:subs_approx_III}
\begin{split}
\left|c(\gamma ){d^{(\ell_{k_j}, \gamma)} \over d} - 1  \right| \leq {\delta \over 9} \left\| f \right\|_{\infty}^{-1}.
\end{split}
\end{equation}

Now, consider the function $\hat{f}_r (x) := f_r (x;W_F^*,p) + W_{\gamma}'/d$.
Since $f_r^{(\ell_{k_j},\gamma)} (\hat{y}_{\ell_{k_j},\gamma})=0$, we can write

\begin{equation*}
\begin{split}
\hat{f}_r & (\hat{y}_\gamma) = \hat{f}_r (\hat{y}_\gamma) - f_r^{(\ell_{k_j},\gamma)} (\hat{y}_{\ell_{k_j},\gamma}) \\
& \leq
p \left| c(\gamma ) {\hat{d}^{(\ell,\gamma)} \over d } - 1 \right| 
+ (1-p) \left| c(\gamma ) {d^{(\ell,\gamma)} \over d}
\int_{0}^{C_\gamma} f(y \hat{y}_{\ell_{k_j},\gamma}) dF^{*(\ell_{k_j},\gamma)}(y)
- \mathbb{E} \left( f(W_F^* \hat{y}_\gamma) \right) \right| \\
&\leq
p \left| c(\gamma ) {\hat{d}^{(\ell,\gamma)} \over d } - 1 \right|
+ (1-p) \left| \left( c(\gamma ) {d^{(\ell,\gamma)} \over d} - 1 \right)
\int_{0}^{C_\gamma} f(y \hat{y}_{\ell_{k_j},\gamma}) dF^{*(\ell_{k_j},\gamma)}(y) \right| \\
&\hspace{1cm}+
(1-p) \left|\int_{0}^{C_\gamma} f(y \hat{y}_{\ell_{k_j},\gamma}) dF^{*(\ell_{k_j},\gamma)}(y)
- \mathbb{E} \left( f(W_F^* \hat{y}_\gamma) \right) \right| \\
&\stackrel{(\ref{eq:subs_approx_I}), (\ref{eq:subs_approx_II}), (\ref{eq:subs_approx_III})}{\leq} {\delta \over 3} + {\delta \over 3}
+ {\delta \over 3} = \delta.
\end{split}
\end{equation*}
Since $\delta$ is arbitrary, it follows that
$$\hat{f}_r (\hat{y}_\gamma) := f_r (\hat{y}_\gamma;W_F^*,p) + W_{\gamma}' = 0, $$
whereby $f_r (\hat{y}_\gamma;W_F^*,p)  < 0$. Recall also that $f_r (\hat{y} -\delta_0;W_F^*,p) > 0$.
The continuity of $f_r$ implies that there is a root in $(0,\hat{y}-\delta_0)$.
But $\hat{y}$ is the smallest positive root of $f_r (x;W_F^*,p)  = 0$ and
$\hat{y}_\gamma \in [0,\hat{y} - \delta_0]$. Therefore, $\hat{y}_\gamma=0$.

But this yields a contradiction as for large $j$ we would have $\hat{y}_{\ell_{k_j},\gamma} < W_{\gamma}'/2d,$
and therefore $$ f_r^{(\ell_{k_j},\gamma)} (\hat{y}_{\ell_{k_j},\gamma})  > W_{\gamma}' /2d >0.$$
\end{proof}

\end{proof}
The following lemma shows that if the weight sequence has power law distribution with exponent between 2 and 3, then the
condition on the derivative of $f_r(x;W_F^*,p)$ that appears in the statement of Theorem~\ref{thm:FinalInfection} is always satisfied.
\begin{lemma} \label{lem:power_law_derivative}
Assume that $({\bf w} (n))_{n \geq 1}$ follows a power law with exponent $\beta \in (2,3)$.
Then $f_r'(\hat{y};W_F^*,p) < 0.$
\end{lemma}
\begin{proof}
From the definition of $f$ we obtain that
\begin{equation*}
\begin{split}
f_r' (x; W_F^*,p)
= -1 + (1-p){r \over x} \ex{e^{-W_F^* x}
{\left( W_F^* x \right)^r \over r!} }.
\end{split}
\end{equation*}
To show the claim it is thus sufficient to argue that
$$ (1-p)r \ex{e^{-W_F^* \hat{y}}
{\left( W_F^* \hat{y} \right)^r \over r!} } < \hat{y} = p + (1-p) \ex {\psi_r (W_F^* \hat{y})}.$$
In turn, it suffices to prove that
\begin{equation}
\label{eq:rEE}
r \ex{e^{-W_F^* \hat{y}}{\left( W_F^* \hat{y} \right)^r \over r!} } < \ex {\psi_r (W_F^* \hat{y})}.
\end{equation}
We set $p_r (x) = e^{-x}x^r/r!$. Furthermore, we set
$g(x):=\ex{ p_r( W_F^* x) }$ and $f(x):=\ex { \psi_r \left( W_F^* x  \right)}$. Then we claim that
$$f(x) > r g(x) \quad \text{for any } x\in(0,1],$$
which is equivalent to~\eqref{eq:rEE}. To see the claim, we will consider the difference $f(x) - rg(x)$ and show that it is increasing with respect to $x$; the statement then follows from $f(0) - rg(0)  = 0$. The derivative with respect to $x$ is
\begin{equation*}
\begin{split}
(f(x)-r g(x))' &= \ex {W_F^* p_{r-1}\left( W_F^* x \right)} + r\left( \ex{ W_F^* p_{r}\left( W_F^* x \right)}
- \ex{ W_F^* p_{r-1}\left( W_F^* x \right)} \right) \\
&= -{r(r-1)\over x}~\ex {p_r\left( W_F^* x \right)} +  {r(r+1) \over x}~\ex {p_{r+1} \left( W_F^* x \right)} \\
&= {r\over x}~ \left( -(r-1)~\ex {p_r\left( W_F^* x \right)} +  (r+1)~\ex {p_{r+1} \left( W_F^* x \right)}\right).
\end{split}
\end{equation*}
Hence, it suffices to show that
$$  (r+1)~\ex {p_{r+1} \left( W_F^* x \right)} > (r-1)~\ex {p_r\left( W_F^* x \right)}, $$
for $x \in (0,1]$.
Note that the probability density function of $W_F^*$ is $(\beta - 1) c w^{-\beta + 1}$, for $w > x_0$; otherwise it is equal to 0.
So we obtain for $j \in \{r,r+1\}$
\begin{equation*} \label{eq:p_r}
\begin{split}
\ex {p_j\left( W_F^* x \right)} &= (\beta - 1) c \int_{x_0}^{\infty} e^{-wx}~{(wx)^j \over j!}~w^{-\beta + 1} dw 
\stackrel{(z=wx)}{=} (\beta - 1){x^{\beta - 2}\over j!}c \int_{x_0}^{\infty} e^{-z} z^{j-\beta + 1} dz.
\end{split}
\end{equation*}
Thereby, it suffices to show that
$$ \int_{x_0}^{\infty} e^{-z} z^{r-\beta + 2} dz > (r-1) \int_{x_0}^{\infty} e^{-z} z^{r-\beta + 1} dz. $$
Applying integration by parts on the integral of the left-hand side we obtain
\begin{equation*} \label{eq:recursion}
\begin{split}
\int_{x_0}^{\infty} e^{-z} z^{r-\beta + 2} dz &= e^{-x_0} x_0^{r-\beta + 2} + (r-\beta + 2)
\int_{x_0}^{\infty} e^{-z} z^{r-\beta + 1} dz \\
& > (r-\beta + 2)  \int_{x_0 }^{\infty} e^{-z} z^{r-\beta + 1} dz \stackrel{(\beta < 3)}{>} (r-1)
\int_{x_0}^{\infty} e^{-z} z^{r-\beta + 1} dz.
\end{split}
\end{equation*}
\end{proof}

\small
\bibliographystyle{plain}
\bibliography{btsrp}

\end{document}